\documentclass[12pt]{amsart}

\usepackage{amsmath, amssymb}
\usepackage{amssymb}
\usepackage{amsmath}
\usepackage{amsthm}
\usepackage{amsfonts}
\usepackage{version}
\usepackage{graphicx}%
\usepackage[T1]{fontenc}
\usepackage{verbatim} 
\usepackage[margin=1 in]{geometry}
\usepackage{amsthm}
 
\providecommand{\U}[1]{\protect\rule{.1in}{.1in}}

\newtheorem{thrm}{Theorem}
\newtheorem{prop}[thrm]{Proposition}

\newtheorem{dfn}[thrm]{Definition}
\newtheorem{lemma}[thrm]{Lemma}
\newtheorem{cor}[thrm]{Corollary}

\newtheorem{rmrk}[thrm]{Remark}
\newtheorem{example}[thrm]{Example}
\numberwithin{equation}{section}

\begin{document}

\newcommand{\lpl}{\left\|}
\newcommand{\rpr}{\right\|_{L^{p}(\mathcal{Q}_r)}}
\newcommand{\rprr}{\right\|_{L^{p}(\mathcal{Q}_{2r})}}
\newcommand{\rpo}{\right\|_{L^{p}(\mathcal{Q}_1)}}

\newcommand{\R}{\mathbb R}
\newcommand{\N}{\mathbb N}
\newcommand{\Q}{\mathcal Q}
\newcommand{\lie}{\mathcal G}
\newcommand{\hN}{\mathcal N}
\newcommand{\D}{\mathcal D}
\newcommand{\A}{\mathcal A}
\newcommand{\B}{\mathcal B}
\newcommand{\sL}{\mathcal L}
\newcommand{\sLi}{\mathcal L_{\infty}}

\newcommand{\eps}{\epsilon}
\newcommand{\al}{\alpha}
\newcommand{\be}{\beta}
\newcommand{\p}{\partial}  

\def\dist{\mathop{\varrho}\nolimits}

\newcommand{\Om}{\Omega}
\newcommand{\om}{\omega}
\newcommand{\half}{\frac{1}{2}}
\newcommand{\e}{\epsilon}
\newcommand{\vn}{\vec{n}}
\newcommand{\X}{\Xi}
\newcommand{\tLip}{\tilde  Lip}
\newcommand{\Span}{\operatorname{span}}

\title[Pointwise Schauder Estimates]{Pointwise Schauder estimates of parabolic equations in Carnot groups}
\author[Price]{Heather Price}
\address{University of Arkansas}
\email{haprice@uark.edu}
\keywords{}
\subjclass{}
\begin{abstract}
Schauder estimates were a historical stepping stone for establishing uniqueness and smoothness of solutions for certain classes of partial differential equations.  Since that time, they have remained an essential tool in the field.  Roughly speaking, the estimates state that the H\"older continuity of the coefficient functions and inhomogeneous term implies the H\"older continuity of the solution and its derivatives.  This document establishes pointwise Schauder estimates for second order ``parabolic'' equations of the form $$\partial_{t}u(x,t)-\sum_{i,j=1}^{m_1} a_{ij}(x,t)X_iX_ju(x,t)=f(x,t)$$
where $X_{1},\ldots ,X_{m_1}$ generate the first layer of the Lie algebra stratification for a Carnot group.  The Schauder estimates are shown by means of Campanato spaces.  These spaces make the pointwise nature of the estimates possible by comparing solutions to their Taylor polynomials. As a prerequisite device, a version of both the mean value theorem and Taylor inequality are established with the parabolic distance incorporated.
\end{abstract}
\maketitle


\section{Introduction}

Schauder estimates are an essential tool in regularity theory for partial differential equations.  Roughly speaking the Schauder estimates state that given a solution to an inhomogeneous equation where the coefficients of the operator as well as the inhomogeneous term are both H\"older continuous, this regularity transmits through the operator to give H\"older continuity of the derivatives of the solution. These estimates were the key to showing uniqueness and smoothness of solutions for certain classes of equations \cite{Nas}.

Juliusz Schauder is credited for the proof in the case of second order linear elliptic equations given in \cite{SC4} and \cite{SC5}.  Though Caccioppoli also had a similar result around the same time, his work was not as detailed \cite{Cac}.  H\"older continuity in the much simpler case of the Laplacian is due to Hopf \cite{Hop} a few years prior to Schauder's result.  Because of the usefulness of the inequality, a common objective of showing these types of estimates for different types of equations under more general conditions arose.  As a result, many methods of proof have emerged. Mentioned here are only a few most relevant to the work of this paper.  A more complete discussion on H\"older estimates and regularity of solutions can be found in \cite[Chapter 6]{GT} for elliptic equations or \cite[Chapter 4]{Li} for parabolic equations.

One method of deriving Schauder estimates depends on having a representation of a fundamental solution, explicitly computing derivatives, and relying on methods of singular integrals to get the results.  This is demonstrated in \cite[Chapter 6]{GT} as well as Chapter 1 of this manuscript.  Another method is by means of the Morrey-Campanato classes, which are equivalent to the H\"older spaces and can be seen in \cite[Chapter 3]{Gia2} as well as \cite[Chapter 4]{Li}.  A third method is based on a scaling argument and approximation of solutions by Taylor polynomials.  It can be found in \cite[Section 1.7]{Si}.  The proof for the Schauder estimates given here has aspects reminiscent of all these methods.

The main result of this paper is a generalization of the classic result in two ways.  The ``pointwise'' result requires only H{\"{o}}lder continuity of the coefficients and inhomogeneous term at a single point in order to get the H{\"{o}}lder continuity of the solution and its derivatives at that same point.  This pointwise nature can be useful when the source term or coefficients are not well behaved everywhere. It also allows an application to the study of nodal sets \cite{Han1}.  The other generalization is the change from the Euclidean setting to the Carnot group setting where derivatives are given by vector fields which may not necessarily commute.  

Comparing the solution to its Taylor polynomial, a technique first popularized by Caffarelli in 1998 in his work on fully nonlinear equations \cite{Caff}, makes possible the pointwise generalization of the estimate.  Caffarelli's approach was generalized to parabolic equations by Wang in \cite{Wa}, and about a decade later, Han used this same method for proving pointwise Schauder estimates for higher order parabolic and elliptic equations in \cite{Han1} and \cite{Han2}.  His interest was in the application to nodal sets.  These results were extended by Capogna and Han \cite{CH} to second order subelliptic linear equations over Carnot groups in 2003.  The proof contained in this paper follows the same method.

Global and local Schauder estimates have been explored in the group setting as well as the more general case of H\"ormander type vector fields. Though this list is not exhaustive, see \cite{BR1}, \cite{SchEst}, \cite{BZ}, \cite{BuR}, \cite{GL}, and \cite{XuC} for more details.  However, the pointwise result contained here seems to be new.

The outline of the paper is as follows.  Section 2 begins with basic definitions related to Carnot groups before proving the Schauder estimates for the sublaplacian in this setting.  Section 3 begins the transition to the parabolic setting.  Definitions regarding the product space of a group and the real line are made.  The main tools such as group polynomials, Campanato classes, and Sobolev spaces are defined, and several lemmas regarding these items are shown.   The proof of the main theorem (stated below) is given in section 4.

Before stating the main result, the operator of interest is given by
\begin{equation}
H_A=\partial_{t}-\sum_{i,j=1}^{m_1}a_{i,j}(x,t)X_{i}X_{j}
\end{equation}
where the vector fields $X_{1},\ldots,X_{m_1}$ generate the first layer of the Lie algebra stratification for a Carnot group, and the matrix $A=(a_{ij})$ is H\"older continuous only at the origin.  Additionally there exist constants $1<\lambda\leq \Lambda< \infty$ such that
\begin{equation}
\lambda|\xi|^2\leq  \  \sum_{i,j=1}^{m_1} a_{ij}(x,t) \xi_i \xi_j  \  \leq\Lambda|\xi|^2 \ \ \mathrm{for \ any}\ \ \xi\in \R^{m_1}.
\end{equation} 
The operator is a non-divergence form similar to the heat equation, but it is not truly parabolic as the title suggests since $m_1$ may be less than the dimension of the space. 

Let $Q$ denote the homogeneous dimension of the group $\mathbb{G}$ and $S_{p}^{2,1}$ denote the Sobolev space containing two spatial derivatives and one time derivative.  The exact definitions for the H\"older and Campanato classes can be found in Section 3.3.

\begin{thrm}\label{Main}
For $Q+2<p<\infty,$ let $u\in S_{p}^{2,1}(\Q_1)$ and $H_Au=f$ in $\Q_1$ with \\$f\in L^{p}(\Q_1).$  Assume, for some 
$\alpha\in(0,1)$ and some integer $d\geq2$, $f\in C_{p,d-2}^{\alpha}(0,0)$ and $a_{ij}\in C_{p,d-2}^{\alpha}(0,0)$.  
Then $u\in C_{\infty,d}^{\alpha}(0,0)$ and 
\begin{equation}
\left\|u\right\|_{\infty,\alpha,d}(0,0)\leq C(\lpl u\rpo + \left\|f\right\|_{p,\alpha,d-2}(0,0))\nonumber
\end{equation}
where $C=C(\mathbb{G},p,d,\alpha,A)>0.$
\end{thrm}

Section 4.1 gives bounds for the heat kernel associated to the constant coefficient equation as well as a few lemmas regarding polynomial expansions of solutions.  These are essential for the constant coefficient a-priori estimates shown in Section 4.2.  These estimates give specific information about the bounds of the Sobolev norm of solutions, which is then used to give a basic version of the Schauder estimates for the constant coefficient equation as a quick corollary.  The corollary is vastly useful because it allows us to transmit information from the polynomials approximating the inhomogeneous term to the polynomial approximating the solution.  The freezing technique is then employed to give a-priori estimates for the non-constant coefficient equation, and a weak version of the Schauder estimates is shown.  (It is weak in the sense that we still assume some amount of regularity on solutions and inhomogeneous term.)  Finally, by comparing the solution to its first order Taylor polynomial and successively applying the a-priori estimates, the final result is obtained.

\section{Carnot Groups}


The setting for the sequel is a special type of Lie group with many structures allowing computations to be done in a fashion similar to the Euclidean setting.  Most essential to the proof of the Schauder estimates is the extensive use of the homogeneity of Carnot groups.  However, care must be taken in the development of ideas such as scaling and distance.  This introduction to Carnot groups aims to make these notions clear and precise while pointing out some of the difficulties of working in these groups, the most obvious of which is the fact that the derivatives do not necessarily commute.

Before commencing, let it be known that in this document all vector fields can be written as linear combinators of standard partial derivatives with smooth coefficient functions.  That is 
$$X=\sum_{k=1}^{n}b_k(x)\partial_{x_k} \,\,\,\,where\,\,\,\, b_k(x) \in C^{\infty}(\R^n).$$  

\begin{dfn}
A Carnot group $\mathbb{G}$ of step $r \geq 1$ is a connected and simply connected nilpotent Lie group whose Lie algebra $\mathfrak{g}$ admits a vector space decomposition into $r$ layers.

$$\mathfrak{g} = V^1 \oplus V^2 \oplus \cdots \oplus V^r$$
having the properties that $\mathfrak{g}$ is graded and generated by $V^1$.  Explicitly, $[V^1, V^j] = V^{j+1}$, $j=1,\ldots,r-1$ and $[V^j,V^r]=0$, $j=1, \ldots, r$.  
\end{dfn}

Let $m_j = \text{dim}(V^j)$ and let $X_{i,j}$ denote a left-invariant basis of $V^j$ where $1 \leq j \leq r$ and $1 \leq i \leq m_j$.  The homogeneous dimension of $\mathbb{G}$ is defined as $Q=\sum_{k=1}^r km_k$.  For simplicity, we will set $X_i = X_{i,1}$.  We call $\{X_{i}\}$ the horizontal vector fields and call their span, denoted $H\mathbb{G}$, the horizontal bundle.  We call $\{X_{i,j}\}_{2\leq j\leq r}$ the vertical vector fields and refer to their span, denoted $V\mathbb{G}$, as the vertical bundle.  Then $\mathfrak{g} = H\mathbb{G} \oplus V\mathbb{G}$.  In fact, the Lie algebra spans the whole tangent space of the group ($T\mathbb{G}=\mathfrak{g}$). 

Because of the stratification of the Lie algebra, there is a natural dilation on $\mathfrak{g}$.  If $X=\sum_{k=1}^r X_{k}$ where $X_{k}\in V^{k}$, then the dilation can be defined by ${\delta_s}(X)={\sum_{k=1}^r s^{k}X_{k}}.$  It is worth noting here that while dilation mappings are defined on the Lie algebra, the mapping $exp\circ\delta_s\circ exp^{-1}$ gives the dilation on the group $\mathbb{G}$.  However, the same notation, $\delta_s,$ will be used for both maps.

Recall the definition of the exponential of a vector field, $X$.  Fix a point $p\in \mathbb{G}$.  Let $\gamma(t)$ be a curve such that $\gamma(0)=p$ and $\frac{d}{dt}\gamma(t)=X_{\gamma(t)}$.  (The existence of such a curve is guaranteed by the theorem of existence and uniqueness of systems of ordinary differential equations.) The exponential map is defined as $exp_p(X)=\gamma(1)$, or more generally, $exp_p(tX)=\gamma(t)$.  

For Lie groups, $p$ is taken as the identity element, and the exponential map provides a means of relating the Lie algebra to the group itself.  
$$exp: \mathfrak{g}\rightarrow \mathbb{G}.$$

And in the special case of Carnot groups, the exponential map is an analytic diffeomorphism, and the Baker-Campbell-Hausdorff formula holds for all $X$ and $Y$ in $\mathfrak{g}$. For a proof see \cite[Theorem 1.2.1]{CoG}.

The Baker-Campbell-Hausdorff formula (BCH) gives a more complete picture of how the exponential map relates the algebra to the group.  Take two vector fields in the Lie algebra, $X$ and $Y$.  The BCH is given by explicitly solving for the vector field $Z$ in the equation $exp(Z)=exp(X)\cdot exp(Y)$.
\begin{eqnarray}\label{BCH}
Z&=&log(exp(X)\cdot exp(Y))\nonumber \\
&=&X+Y+\frac{1}{2}[X,Y]+\frac{1}{12}\left[X, [X,Y]\right]-\frac{1}{12}\left[Y, [X,Y]\right] + \cdots 
\end{eqnarray}
The formula continues with higher order commutators.  For nilpotent Lie groups, it is clear the summation will eventually terminate. 

\begin{dfn}
Let $\{ X_1,\ldots,X_{n}\}$ be a basis for a nilpotent Lie algebra $\mathfrak{g}$, and consider a map
$$\Psi:\R^n\rightarrow \mathbb{G}$$
$$\Psi(s_1,\ldots,s_n)=exp(s_1X_1+\ldots+s_nX_n).$$
The coordinates given by the map $\Psi$ are called exponential coordinates or canonical coordinates of the first kind.
\end{dfn}
Canonical coordinates of the second kind are defined similarly by taking $$\Psi(s_1,\ldots,s_n)=exp(s_1X_1)\cdots exp(s_nX_n).$$

The foundation has now been laid to give a few examples of Carnot groups.

\begin{example}  The Heisenberg group, $\mathbb{H}^n$, is a step 2 Carnot group whose underlying manifold is $\R^{2n+1}$.  Taking $x,x' \in \R^{2n}$ and $t,t' \in \R$, the group operation is given by
$$(x,t)\cdot(x',t')=(x+x',t+t'+2\sum_{i=1}^{n}(x_i'x_{n+i}-x_ix_{n+i}')).$$
The vector fields below form a left-invariant basis for the Lie algebra, $\mathfrak{h}=V^1 \oplus V^2.$
$$X_i=\partial_{x_i}-\frac{1}{2}x_{n+i}\partial_t ,  \,\,\,\,\, X_{i+n}=\partial_{x_{n+i}}+\frac{1}{2}x_{i}\partial_t \,\,\,\,for \,\,\,i=1,\ldots,n$$
$$and\,\,\,\,T=\partial_t$$
The horizontal bundle is given by $V^1=span\{X_1,\ldots,X_{2n}\},$ and the vertical bundle is then $V^2=span\{T\}$ leading to a homogeneous dimension of $2n+2$.  
\end{example}

\begin{example}The Engel group, $\mathbb{K}_3,$ is an example of a step 3 Carnot group with a homogeneous dimension of $7$.  See \cite{CoG} and \cite{DGN} for more information.  This group has an underlying manifold of $\R^4,$ and the group operation is given by
$$x\cdot x'=(x_1+x_1', \,\,x_2+x_2',  \,\,x_3+x_3'+A_3,  \,\,x_4+x_4'+A_4)$$
where
$$A_3=\frac{1}{2}(x_1x_2'-x_2x_1')$$
and
$$A_4=\frac{1}{2}(x_1x_3'-x_3x_1')+\frac{1}{12}(x_1^2x_2'-x_1x_1'(x_2+x_2')+x_2(x_1')^2).$$
The Lie algebra can be graded as $\mathfrak{g}=V^1 \oplus V^2 \oplus V^3$ by letting $V^1=span\{X_1,X_2\}$, $V^2=span\{X_3\}$, and $V^3=span\{X_4\}.$  Using the Baker-Campbell-Hausdorff formula, expressions for the vector fields can be found.
$$X_1=\partial_{x_1}-\frac{x_2}{2}\partial_{x_3}-\left(\frac{x_3}{2}+\frac{x_1x_2}{12}\right)\partial_{x_4}$$
$$X_2=\partial_{x_2}+\frac{x_1}{2}\partial_{x_3}+\frac{x_1^2}{12}\partial_{x_4}$$
$$X_3=\partial_{x_3}+\frac{x_1}{2}\partial_{x_4}$$
$$X_4=\partial_{x_4}$$
Notice that $[X_1,X_2]=X_3$ and $[X_1,X_3]=X_4$ and all other commutators are trivial.

\end{example}

For demonstrative purposes, return to the first Heisenberg group, $\mathbb{H}^1$.  This group is of great interest and widely studied not only because of its appearance in applications but also because there are only two 3 dimensional simply connected nilpotent Lie groups: $\mathbb{H}^1$ and $\R^3$.

It is simple to check the left-invariance of the vector fields for $\mathbb{H}^1$. Let $f(x,y,t)$ be a left translation by (a,b,c).
$$f(x,y,t)=\left(x+a,y+b,t+c+1/2(ay-bx)\right)$$
The differential is 
\begin{equation}
df= \begin{pmatrix} 1 & 0 & 0 \\ 0 & 1  & 0 \\ -b/2 & a/2 & 1 \end{pmatrix}.
\end{equation}
Consider first the vector field $X_1=\partial_x-(y/2) \partial_t=\begin{pmatrix} 1 \\ 0 \\ -y/2 \end{pmatrix} $.  If we are to first take the $X_1$ derivative at a point $p=(x,y,t)$ and then apply the left translation in the tangent space, we get $df\cdot  X_1= \partial_x +(-b/2-y/2)\partial_t$.  On the other hand, if we left translate the point and then find the derivative at the translated point, we get $X_1 \circ f(x,y,t) =\partial_x+(-b/2-y/2)\partial_t.$ Therefore $f_*X_1=X \circ f$ showing $X_1$ is left-invariant.  The same can be done for the other vector fields.  A nice derivation of the vector fields is given in \cite{CaD}.

For $\mathbb{H}^1$, the underlying manifold is $\R^3$, but $X_1$ and $X_2$ do not span all of the tangent space.  The commutator, $T=\partial_t$, recovers the missing direction. This group is said to satisfy \textit{H\"{o}rmander's condition}.  In fact, Carnot groups in general satisfy H\"ormander's condition, meaning that the basis of vector fields along with all of their commutators up to some finite step will span the entire tangent space.  This property is essential when considering Carnot groups as sub-Riemannian manifolds because it ensures that any two points in the group can be connected by a path that lies entirely in the span of $V^1$.  This type of path is referred to as a horizontal path.  To be more precise, stated below is the fundamental theorem.

\begin{thrm} \label{chow} (Chow's Theorem)
If a smooth distribution satisfies H\"{o}rmander's condition at some point p, then any point q which is sufficiently close to p can be joined to p by a horizontal curve.
\end{thrm}

Because the $L^2$ energy required to travel only along directions in vertical bundle is infinite, the horizontal vector fields give the so called ``admissible" directions.  Getting from one point to another requires traveling along a horizontal curve, and Chow's theorem alleviates any concern that we might not be able to find an appropriate path.  We will always be able to get there by moving along horizontal curves so long as the vector fields satisfy H\"ormander's finite rank condition.  

Carnot groups can also be viewed as sub-Riemannian manifolds, and it is always possible to define $g$ a Riemannian metric with respect to which the $V^j$ are mutually orthogonal.  A curve $\gamma:[0,1] \to \mathbb{G}$ is horizontal if the tangent vector $\gamma'(t)$ lies in $V^1$ for all $t$.  The Carnot-Carath\'eodory distance (CC-distance) can now be defined.

\begin{dfn} Let $p,q \in \mathbb{G}.$
$$d_{cc}(p,q) = \inf\int_0^1 \left(\sum_{i=1}^{m_1} \left \langle \gamma'(t),X_i|_{\gamma(t)}\right \rangle ^2_g \, dt\right)^{1/2},$$
where the infimum is taken over all horizontal curves $\gamma$ such that $\gamma(0) = p$, $\gamma(1) = q$ and $\left \langle \cdot,\cdot \right \rangle_g$ denotes the left invariant inner product on $V^1$ determined by the metric $g$.
\end{dfn}

Chow's Theorem gives the existence of the horizontal curve, $\gamma.$  Another consequence of his work is that the CC-distance is finite for connected groups.

The CC-distance is not a true distance in that it lacks the triangle inequality.  However it does satisfy the quasi-triangle inequality meaning there exists a positive constant, $A$, depending on the group $\mathbb{G}$ such that 

$$d_{cc}(x,y)\leq A(d_{cc}(x)+d_{cc}(y)).$$

We will also make use of the distance defined by the gauge norm. Let $x_{i,k}$ be the coordinates for a point $x\in \mathbb{G},$ then
\begin{equation}
|x|^{2r!}=\sum_{k=1}^r\sum_{i=1}^{m_1} |x_{i,k}|^{2r!/k}.
\end{equation}
For $x,y \in \mathbb{G}$, we then let $d(x,y)=|xy^{-1}|$ as defined above.

The gauge distance is equivalent to the CC-distance but has the advantage of being a Lipschitz function.  We say they are equivalent due to the fact that there exists a constant, $a(\mathbb{G}),$ dependent on the group such that $$a^{-1}d_{cc}(x,y)\leq d(x,y) \leq a d_{cc}(x,y).$$  For a proof, as well as other properties of these metrics see \cite{NSW}.


\section{The Parabolic Setting}

\subsection{ Parabolic Distance, Balls, and Cylinders} 

Throughout this paper, the relevant space is $\mathbb{G}\times \R$, where $x$ is reserved as a space variable in $\mathbb{G}$ and $t$ is thought of as a time variable in $\R$.  This is also a Carnot group where time derivatives appear in the first layer of the stratification. However, homogeneity considerations of the operator would dictate the time derivative should have weight 2, and a different dilation mapping would be needed rather the natural one given previously.  An alternative viewpoint is the one given by Rothschild and Stein.  In \cite{RS}, they referred to this situation (where the algebra is generated by a vector field in the second layer, $X_o,$ in addition to the vector fields $X_1, \ldots, X_{m_1}$ in the first layer) as a type II stratified group.  

This document mostly keeps to the viewpoint of separately dealing with $\mathbb{G}$ and $\R$ to get results on the product space.  Therefore we define a new dilation mapping
$\delta'_s:\mathbb{G}\times \R\rightarrow \mathbb{G}\times \R$ given by
$$\delta'_s(x,t)=(\delta_s(x),s^2t).$$

With that being said, the parabolic distance is defined.
\begin{dfn} 
Let $(x,t),(y,s)\in \mathbb{G}\times\R$.  The parabolic distance is
$$d_{p}((x,t),(y,s))=(d(x,y)^{2}+\left|t-s\right|)^{1/2}.$$ 
\end{dfn}

Because $\mathbb{G}$ is stratified, one can always find such a homogeneous norm $d(x,y)$ and a dilation mapping $\delta_s(x)$ such that $d(\delta_s(x),\delta_s(y))=sd(x,y)$.  Both the gauge distance and CC-distance fulfill this requirement.

 This is the appropriate distance for the dilation chosen since 
$$d_{p}(\delta'_{s}(x,t))=(d(\delta_{s}x)^{2}+\left|s^{2}t\right|)^{1/2}
=sd_{p}(x,t)$$ as desired.  For sake of simplicity, we will use $\left|(x,t)(y,s)^{-1}\right|$ to denote the parabolic distance between the points $(x,t)$ and $(y,s),$ and $|(x,t)|$ to be the parabolic distance between $(x,t)$ and the origin.  

Using the quasi-triangle inequality for $d(x,y)$ on the group (with constant $A$), the analogue can be shown for $d_p$ with the same constant.  
\begin{prop}There exists a constant, $A$ depending on the group, $\mathbb{G}$, such that for all points $(x,t),(y,s),(z,\tau)\in \mathbb{G}\times \R$ the inequality holds.
\begin{equation}\label{quasi}
|(x,t)(y,s)^{-1}|\leq A\left(|(x,t)(z,\tau)^{-1}|+|(z,\tau)(y,s)^{-1}|\right).
\end{equation}
\end{prop}

\begin{proof}
If we can show that 
$$d(x,y)^2+|t-s|\leq C\left(d(x,z)^2+|t-\tau|+d(z,y)^2+|\tau-s|\right),$$
then we are done since
\begin{eqnarray*}
d(x,y)^2+|t-s|&\leq&C\left(d(x,z)^2+|t-\tau|+d(z,y)^2+|\tau-s|\right)\\
&\leq & C(d(x,z)^2+|t-\tau|+2\sqrt{d(x,z)^2+|t-\tau|}\sqrt{d(z,y)^2+|\tau-s|}\\
&&+d(z,y)^2+|\tau-s|)\\
&=&C\left(\sqrt{d(x,z)^2+|t-\tau|}+\sqrt{d(z,y)^2+|\tau-s|}\right)^2.\\
\end{eqnarray*}
Taking the square root of both sides, would give the desired inequality
$$\sqrt{d(x,y)^2+|t-s|}\leq \sqrt{C}\left(\sqrt{d(x,z)^2+|t-\tau|}+\sqrt{d(z,y)^2+|\tau-s|}\right).$$

If $A^2>1$ in the quasi-triangle inequality for $d(x,y),$ we get
\begin{eqnarray*}
d(x,y)^2+|t-s| &\leq&  A^2 d(x,z)^2+|t-\tau|+A^2 d(z,y)^2+|\tau-s| \\
&=& A^2\left( d(x,z)^2+|t-\tau|+d(z,y)^2+|\tau-s|\right) \\
&&+(1-A^2)|t-\tau|+(1-A^2)|\tau-s|\\
&\leq& A^2 \left( d(x,z)^2+|t-\tau|+d(z,y)^2+|\tau-s|\right) .
\end{eqnarray*}

If $A^2<1,$ we actually have a true triangle inequality by adding positive additional terms
$(1-A^2)d(x,z)^2$ and $(1-A^2)d(z,y)$ to the right hand side.
\end{proof}

We also use the following form of reverse triangle inequality,
\begin{equation}\label{rev}
|(x,t)|-A|(y,s)|\leq A|(x,t)(y,s)^{-1}|.
\end{equation}

\begin{dfn}
Choosing $d(x,z)$ to be the gauge distance defined earlier, the set 
$$\Q_{r}(x,t)=\left\{(z,\tau):\left|t-\tau\right|<r^{2} \,\,\mathrm{and}\,\, d(x,z)<r \right\}$$
 is known as the parabolic cylinder.
\end{dfn}

Let $B_r(x)$ denote the CC-ball on $\mathbb{G}$, then it is easy to see that
$\left|\Q_r(x,t)\right|= r^{Q+2}\left|B_1(0)\right|$ 
where the measures indicated are Lebesgue measures.  A simple computation gives $|B_r(x)|=|B_1(0)|r^Q$, and consequently $\left|\Q_r(x,t)\right|= r^{Q+2}\left|B_1(0)\right|$.  The Jacobian determinant of $\delta_r:\mathbb{G}\rightarrow \mathbb{G}$ is simply $r^Q$, so the calculation follows by a change of variable.  See \cite{BLU} for this calculation and other properties of the CC-balls.  Properties of the parabolic balls and parabolic cylinders are nicely explained in \cite{SchEst}.

Sometimes it is convenient to use the notation $ \Q_{r}(x,t)=B_r(x)\times \Lambda_r(t)$ where $\Lambda_r(t)=\left\{\tau\in\R :\ \ |t-\tau|<r^2\right\}.$

We can also define a parabolic ball to be the set 
$$B_p((x_o,t_o),r):=\left\{(y,s): |(x_o,t_o)(y,s)^{-1}|<r\right\}.$$
Its size is comparable to the parabolic cylinder, but they are not the same set.


\subsection{Group Polynomials}

In this section, definitions and terminology regarding group polynomials are given followed by several results regarding Taylor polynomials. For more details see \cite{Hardy}.

Let $l\in \N$ and consider a multi-index, $I=[(i_1,k_1),(i_2,k_2),\ldots,(i_l,k_l)]$, where $1\leq k_j \leq r$ and $1\leq i_j \leq k_{m_j}.$ Derivatives of a smooth function, $f$, defined in $\mathbb{G}$ will be denoted as 
$$X^I f=X_{i_1,k_1}X_{i_2,k_2}\ldots X_{i_l,k_l}f$$
where the order of the derivative is $|I|=\sum_{j=1}^l k_j.$
Throughout this paper, we are only concerned with derivatives appearing in the first layer of the stratification meaning the order will simply be the number of vector fields applied to the function.
\begin{dfn}
A group polynomial on $\mathbb{G}\times \R$ is a function that can be expressed in exponential coordinates as
$$P(x,t)=\sum_{I,\beta}a_{I,\beta}x^I t^{\beta}$$ 
where $I=(i_{j,k})_{j=1\ldots m_k}^{k=1\ldots r}$ and $\beta$ and $a_{I,\beta}$ are real numbers, and 
$$x^I=\prod_{j=1\ldots m_k;k=1\ldots r}x_{j,k}^{i_{j,k}},$$
or equivalently,
$$P(x,t)=\sum_{d}a_{I,\beta}(x,t)^{d}$$
where $(x,t)^d=x^I t^{\beta}$ and $|I|+2\beta=d$.
\end{dfn}

The homogeneous degree of the monomial $x^I$ is given by $$|I|=\sum_{k=1}^r \sum_{j=1}^{m_k} ki_{j,k},$$ and the homogeneous degree of $(x,t)^d$ is (as expected) $d=|I|+2\beta$.  We will notate the set of polynomials of homogeneous degree not exceeding $d$ as $\mathcal{P}_d$.  To avoid tedious language, the homogeneous degree will simply be called the degree of the polynomial.

\begin{dfn} If $h\in C_{0}^{\infty}(\mathbb{G}\times\R)$ and $k$ is a positive integer, then the $kth$ order Taylor polynomial $P_k$ of $h$ at the origin is the unique polynomial of homogeneous order at most $k$ such that $$X^{I}D_{t}^{l}P_{k}(0,0)=X^{I}D_{t}^{l}h(0,0)$$ for all $|I|+2l\leq k.$
\end{dfn}
We will need several results regarding Taylor polynomials and their remainders.  The upcoming three of which appear in \cite[pp.33-35]{Hardy}.

\begin{thrm}\label{SMVT} \textbf{(Stratified Mean Value Theorem)}
Let $\mathbb{G}$ be stratified.  There exists constants $c>0$ and $b>0$ such that for every $f\in C^{1}$ and for all $x,y \in \mathbb{G}$,
$$|f(xy)-f(x)|\leq c|y| \sup\limits_{\stackrel{|z|\leq b|y|}{1\leq j \leq m_1}} |X_{j}f(xz)|$$
where $X_{j}$ is in the first layer of the stratification, and $|\cdot|$ is a homogeneous norm on $\mathbb{G}$.
\end{thrm}

\begin{thrm}\label{STI}\textbf{(Stratified Taylor Inequality)}
Let $\mathbb{G}$ be stratified.  For each positive integer $k$, there exists a constant $c_k>0$ such that for every $f\in C^{k}$ and for all $x,y\in \mathbb{G}$, 
$$|f(xy)-P_{x}(y)|\leq c_{k}|y|^{k} \sup\limits_{\stackrel{|z|\leq b^{k}|y|}{|I|=k}} |X^{I}f(xz)-X^{I}f(x)|$$
where $P_{x}(y)$ is the left Taylor Polynomial of $f$ at $x$ of homogeneous degree $k$ and $b$ is as in the Stratified Mean Value Theorem.
\end{thrm}

\begin{cor}\label{CSTI} If $k\geq1$, then there exists positive constants $C_k$ and $b$ (independent of $k$) such that for every $f\in C^{k+1}(\mathbb{G})$ and all $x,y\in \mathbb{G}$ there holds $$|f(xy)-P_{x}(y)|\leq C_{k}|y|^{k+1} \sup\limits_{\stackrel{|z|\leq b^{k}|y|}{|I|=k+1}} |X^{I}f(xz)|$$ where $P_{x}(y)$ is the $kth$ order left Taylor polynomial of $f$ at the point $x$.
\end{cor}

Now results analogous to those above will be provided, incorporating the time derivatives and the parabolic distance.  To this end, we will first define some special classes of functions.  We will say that $f\in C_*^0(\Omega)$ if $f$ is continuous on the open set $\Omega \in \mathbb{G}\times\R$ with respect to the parabolic distance, and we define  
$$C_*^1(\Omega)\!=\! \left\{f\in C_*^0(\Omega):\ \!\! X_if \in C_*^0(\Omega)\ \mathrm{ for }\ i=1,\ldots,m_1 \ \mathrm{ and }\  \frac{|f(x,t)-f(x,\tau)|}{|t-\tau|^{1/2}}<\infty,\ t\neq \tau \right\}.$$

For any positive integer $k\geq2$, we define $$C_*^k(\Omega)=\left\{f\in C_*^{k-1}(\Omega):\  X_if \in C_*^{k-1}(\Omega)\ \mathrm{ for }\ i=1,\ldots,m_1 \ \mathrm{ and }\  D_tf\in C_*^{k-2}(\Omega) \right\}.$$

Roughly speaking, a function in the set $C_*^k(\Omega)$ will have continuous derivatives up to order $k$ (with respect to the parabolic distance) as well as ``half'' derivatives in the $t$ variable up to order $k-1$ in keeping with the idea that one time derivative is equivalent to two spatial derivatives.

\begin{lemma}\label{PMVT}
Suppose $g\in C_*^{1}(\mathbb{G}\times\R)$, then for all $(y,s)\in \mathbb{G}\times\R/(0,0)$ there exists a positive constant $C$ such that
\begin{eqnarray*}
|g(y,s)-g(0,0)|&\leq& C|(y,s)|\sup\limits_{\stackrel{|(z,\tau)|\leq b|(y,s)|}{i,=1,\ldots,m_1}}|X_ig(z,\tau)|).
\end{eqnarray*}
\end{lemma}

\begin{proof}
Notice first that $|g(y,s)-g(0,0)|\leq |g(y,s)-g(y,0)|+|g(y,0)-g(0,0)|.$

The first term can be estimated using the ``half'' derivative while the second term requires the Stratified Mean Value Theorem.  Together we get the following:
$$|g(y,s)-g(0,0)|\leq C_1|s|^{1/2}+ C |y|\sup\limits_{\stackrel{|z|\leq b|y|}{i,=1,\ldots,m_1}}|X_i g(z,0)|$$ where $0<\tau<s$.

Here $|s|$ is referring to the Euclidean distance and $|y|$ is any homogeneous norm on $\mathbb{G}$.  
Clearly, $|s|^{1/2}\leq |(y,s)|$ and $|y|\leq |(y,s)|$ giving
\begin{eqnarray*}
|g(y,s)-g(0,0)|&\leq& C |(y,s)|\sup\limits_{\stackrel{|z|\leq b|y|}{i,=1,\ldots,m_1}}|X_i g(z,0)|
\end{eqnarray*}
Taking the appropriate supremum gives the conclusion.
\end{proof}

Now we give a parabolic version of the Stratified Taylor Inequality.
\begin{thrm}\label{PSTI}
Suppose $\mathbb{G}$ is stratified.  For every positive integer k, there exists positive constants $C$ and $b$ depending on $k$ such that for all $f\in C_*^{k}(\mathbb{G}\times\R)$ and all $(x,t),(y,s)\in \mathbb{G}\times\R$, 
$$|f(y,s)-P_k(y,s)|\leq C_k|(y,s)|^{k} \sup\limits_{\stackrel{|(z,\tau)|\leq b^k|(y,s)|}{|I|+2l=k}}|X^{I}D_{t}^{l}f(z,\tau)-X^{I}D_t^{l}f(0,0)|$$
where $P_k(y,s)$ is the $kth$ order Taylor polynomial about the origin.
\end{thrm}

\begin{proof}
The method of proof is similar to the proof of the Stratified Taylor Inequality in \cite{Hardy}.

Fix a $(y,s)\in \mathbb{G}\times\R$ and let $g(x,t)=f(x,t)-P_k(y,s).$ 
By definition of the Taylor Polynomial for all $|I|+2l\leq k$,
$$X^ID_t^lg(0,0)=X^ID_t^lf(0,0)-X^ID_t^lP_k(y,s)=0.$$

By induction on $n$, we will show for all $0<n\leq k$ 
\begin{equation}\label{hy}
|X^JD_t^pg(y,s)|\leq C_n|(y,s)|^{n}\sup\limits_{\stackrel{|(z,\tau)|\leq b^{n}|(y,s)|}{|I|+2l=k}}|X^{I}D_{t}^{l}f(z,\tau)-X^{I}D_t^{l}f(0,0)|
\end{equation}
where $|J|+2p=k-n.$ The case $n=k$ will give the conclusion.

To begin, we will see the case $n=0$ is trivial since $|J|+2p=k$ and
\begin{eqnarray*}
|X^JD_t^pg(y,s)|&=&|X^JD_t^pf(y,s)-X^JD_t^pP_k(y,s)|\\
&=&|X^JD_t^pf(y,s)-X^JD_t^pf(0,0)|\\
&\leq&\sup\limits_{\stackrel{|(z,\tau)|\leq |(y,s)|}{|I|+2l=k}}|X^{I}D_{t}^{l}f(z,\tau)-X^{I}D_t^{l}f(0,0)|.
\end{eqnarray*}

Suppose \eqref{hy} is true for $n=k-1,$ and consider the case $n=k$. Using Lemma \ref{PMVT}, we have 
\begin{eqnarray*}
|g(y,s)-g(0,0)|&=&|g(y,s)|\\
&\leq& C|(y,s)|\sup\limits_{\stackrel{|z|\leq b|y|}{i,=1,\ldots,m_1}}|X_i g(z,\tau)|.
\end{eqnarray*}

Using the $n=k-1$ case,
\begin{eqnarray*}
|g(y,s)|&\leq& C|(y,s)|^k \sup\limits_{\stackrel{|(z,\tau)|\leq b |(y,s)|}{|I|+2l=k}}|X^ID_t^l g(z,\tau)|\\
&=& C|(y,s)|^k \sup\limits_{\stackrel{|(z,\tau)|\leq b |(y,s)|}{|I|+2l=k}}|X^ID_t^l f(z,\tau)-X^ID_t^lf(0,0)|.
\end{eqnarray*}
\end{proof}

\begin{cor}\label{CPSTI}
Suppose $\mathbb{G}$ is stratified.  For every positive integer $k$, there exists positive constants $C$ and $b$ depending on $k$ such that for all $f\in C_*^{k+2}(\mathbb{G}\times\R)$ and all $(x,t),(y,s)\in \mathbb{G}\times\R$,
\begin{eqnarray*}
|f(y,s)-P_k(y,s)|&\leq& C|(y,s)|^{k+1}\sup\limits_{\stackrel{|(z,\tau)|\leq b|(y,s)|}{|I|+2l=k+1}}|X^{I}D_{t}^{l}f(z,\tau)|
\end{eqnarray*}
where $P_k(y,s)$ is the $kth$ order Taylor Polynomial about the origin.
\end{cor}

\begin{proof}
Simply apply Lemma \ref{PMVT} to Theorem \ref{PSTI}.
\end{proof}


\subsection{Campanato Spaces and Embeddings}

The Pointwise Schauder estimates are proved by means of the Morrey-Campanato norms.  In this section we will establish the relationship between the Campanato classes and the Folland-Stein H{\"{o}}lder spaces. 

First we will adapt the Folland-Stein H{\"{o}}lder norm to include the parabolic distance as
$$||f||_{\Gamma^{\alpha}(x_o,t_o)}=\sup\limits_{\stackrel{(x,t)\neq(x_o,t_o)}{(x,t)\in\Omega}}\frac{|f(x,t)-f(x_o,t_o)|}{|(x,t)(x_o,t_o)^{-1}|^{\alpha}}$$
where $\Omega$ is an open set in $\mathbb{G}\times\R$ and $0<\alpha<1$.

$f:\Omega\rightarrow\R$ is in the H{\"{o}}lder Space $\Gamma^{\alpha}(\Omega)$ if and only if for all $(x,t)\in\Omega$ there exists a constant C such that $||f||_{\Gamma^{\alpha}(x,t)}\leq C(f,\Omega,\alpha)< \infty,$ and define
$$||f||_{\Gamma^{\alpha}(\Om)}=\sup\limits_{\Om}|f|+\sup\limits_{\stackrel{(x,t)\neq(y,s)}{(x,t),(y,s) \in\Omega}}\frac{|f(x,t)-f(y,s)|}{|(x,t)(y,s)^{-1}|^{\alpha}}$$
where $\Omega$ is an open set in $\mathbb{G}\times\R$ and $0<\alpha<1$.  

$$||f||_{\Gamma^{k+\alpha}(\Om)}=\sum_{j=0}^{k}\sup\limits_{\stackrel{|I|+2l=j}{\Om}}|X^ID_t^lf|+\sup\limits_{\stackrel{(x,t)\neq(y,s)\in \Om}{|I|+2l=k}}\frac{|X^ID_t^lf(x,t)-X^ID_t^lf(y,s)|}{|(x,t)(y,s)^{-1}|^{\alpha}}$$

We can also define the local H{\"{o}}lder spaces.
$$\Gamma^{\alpha}_{loc}(\Omega)=\left\{f:\Omega\rightarrow\R | \ gf\in\Gamma^{\alpha}(\Omega)\ \ for \ \ some \ \ g\in C^{\infty}_{o}(\Omega)\right\}.$$ 

We now give the definition of the Campanato Spaces and Morrey-Campanato norm as in \cite{Hardy}, and \cite{CH}.  These classes are equivalent to the H\"older classes.  Recall that $\mathcal{P}_d$ refers to the set of group polynomials of homogeneous degree less than or equal to $d$. 

\begin{dfn}
Suppose $\Omega$ is open in $\R^{n+1}$, $\alpha\geq0$, $d\in\mathbb{N}$, and $1\leq p\leq \infty.$  Then $C^{\alpha}_{p,d}(x_o,t_o)$ is the set of all functions $u\in L^p_{loc}(\Q_1(x_o,t_o))$ such that 
$$[u]_{p,\alpha,d}(x_o,t_o)=\sup\limits_{0<r<1}\inf\limits_{P\in \mathcal{P}_d}\ r^{-\alpha}\left(\frac{1}{|\Q_r(x_o,t_o)\cap \Omega|}\int_{\Q_r(x_o,t_o)\cap \Omega}|u-P|^{p}(x,t)dxdt\right)^{1/p}<\infty.$$
If $p=\infty$, the $L^{\infty}$ norm should be used in place of the $L^{p}$ norm.
\end{dfn}
It is worth noting that if the polynomial $P$ exists, then it is unique, and consequently, when $d=0$ the classical definition of the space of bounded mean oscillation is recovered in that the constant, $P$, would be the average of $u$ over $\Q_r(x_o,t_o)\cap \Omega$.

While the above is only a semi-norm, we will define and make use of the following true norm.
\begin{equation}
\left\|u\right\|_{p,\alpha,d}(x_o,t_o)= \sum_{|I|+2l=0}^d|X^ID_t^lu(x_o,t_o)|+[u]_{p,\alpha,d}(x_o,t_o).
\end{equation}

For the purposes of this paper, we will select $(x_o,t_o)=(0,0),$ and all results will hold for a general point by left invariance.  

\begin{thrm}\label{G}
Suppose $\alpha>0,$ $1\leq p \leq \infty,$ and $d\in \N$.  If $f\in C^\al_{p,d}(x_o,t_o)$, then for all $k<\al$ ($|I|+2l=k$), $X^ID_t^lf$ is continuous in $\Q_r(x_o,t_o)$ and $X^ID_t^lf\in \Gamma^{\al-k}(\Q_r(x_o,t_o)).$
\end{thrm}

The proof given in \cite[Proposition 5.17]{Hardy} is only for a stratified group, but it works just the same for the parabolic setting because it relies only on the Mean Value Theorem and Taylor inequality, which we have shown in the previous section.  The original proof is actually due to Krantz \cite{K1} and \cite{K2}.  The case $d=0$ will be shown here in a slightly different manner below and depends on a geometric property of the cylinders.   

\begin{dfn}
An open set $\Omega \in \R^{n+1}$ is said to have the the A-property if there exists a constant $A>0$ such that $|\Q_r(x,t)\cap\Omega|\geq A|\Q_r(x,t)|.$
\end{dfn}

\begin{rmrk}
Parabolic cylinders satisfy the A-property.  This is easy to see using the fact that Carnot-Carath{\'{e}}odory balls have the analogous A-property \cite{Cap}, meaning that there exists a positive constant $A'$ such that $|B_r(x)\cap B_R(x_o)|\geq A'|B_r(x_o)|$. This geometric property ensures there are no infinitely sharp cusps on the boundary.
We see that
\begin{eqnarray*}
|\Q_r(x,t)\cap \Q_R(x_o,t_o)|&=&|(B_r(x)\times \Lambda _r(t)) \cap (B_R(x_o)\times \Lambda_R(t_o))| \\
&=&|(B_r(x) \cap B_R(x_o)) \times (\Lambda _r(t) \cap \Lambda_R(t_o))| \\
&=&|B_r(x) \cap B_R(x_o)||\Lambda _r(t) \cap \Lambda_R(t_o)| \\
&\geq& A'C|B_r(x)| |\Lambda _r(t)|\\
&=& A|B_r(x)\times \Lambda _r(t)|\\
&=& A|\Q_r(x,t)|.
\end{eqnarray*}
\end{rmrk}

\begin{prop}\label{mce1}If an open set $\Omega \in \R^{n+1}$ has the A-property, then 
$C_{p,0}^{\alpha}(\Omega)\subset\Gamma^{\alpha}(\Omega))$.
\end{prop}

\begin{proof}
The proof is similar to the Euclidean case in \cite{Gia} given the appropriate definition of the A-property as above.  Throughout this manuscript, $diam(\Omega)$ will be the diameter of $\Omega$ with respect to the parabolic distance, and note that on compact sets, for every $\alpha$, there is a positive constant $R_o$ such that $0<\alpha r<R_o.$
The average of $u$ over a parabolic cylinder will be
$$u_{(x,t),r}=\int_{\Q_r(x,t)} \!\!\!\!\!\!\!\!\!\!\!\!\!\!\!\!\!\!- \,\,\,\,\,\,\,\,\,\,\,u(z,\tau)dzd\tau=\frac{1}{|\Q_r(x,t)|}\int_{\Q_r(x,t)}u(z,\tau)dzd\tau.$$
Given this notation, we have $u\in C_{p,0}^{\alpha}(\Omega) $ if and only if 
$$\int_{\Q_r(x,t)\cap \Omega}\!\!\!\!\!\!\!\!\!\!\!\!\!\!\!\!\!\!\!\!\!\!\!\! - \,\,\,\,\,\,\,\,\,\,\,\,\,\,\,\,\,|u(z,\tau)-u_{(x,t),r}|^p dzd\tau \leq Cr^{\alpha p}$$
for all $(x,t)\in \Omega$ and all $0<r<min\left\{R_o,diam(\Omega)\right\}$.
Now, we will fix $r$ and $R$ so that $0<r<R<min\left\{R_o,diam(\Omega) \right\}$, and see that
$$|u_{(x,t),r}-u_{(x,t),R}|^p\leq C\left[ \int_{\Q_r(x,t)\cap \Omega}\!\!\!\!\!\!\!\!\!\!\!\!\!\!\!\!\!\!\!\!\!\!\!\! - \,\,\,\,\,\,\,\,\,\,\,\,\,\,\,\,\,|u(z,\tau)-u_{(x,t),r}|^p dzd\tau + \int_{\Q_r(x,t)\cap \Omega}\!\!\!\!\!\!\!\!\!\!\!\!\!\!\!\!\!\!\!\!\!\!\!\!- \,\,\,\,\,\,\,\,\,\,\,\,\,\,\,\,\,|u(z,\tau)-u_{(x,t),R}|^p dzd\tau \right]$$
$$\leq C\left[r^{\alpha p} + \frac{|\Q_R(x,t)\cap \Omega|}{|\Q_r(x,t)\cap \Omega|}\frac{1}{|\Q_R(x,t)\cap \Omega|}\int_{\Q_R(x,t)\cap \Omega}|u(z,\tau)-u_{(x,t),R}|^p dzd\tau\right]$$
$$\leq C\left[r^{\alpha p} + \frac{|\Q_R(x,t)\cap \Omega|}{|\Q_r(x,t)\cap \Omega|}R^{\alpha p} \right].$$
Using the A-property and $|\Q_R(x,t)\cap \Omega|\leq |\Q_R(x,t)|,$ we have
$$\frac{|\Q_R(x,t)\cap \Omega|}{|\Q_r(x,t)\cap \Omega|}\leq \frac{R^{Q+2}}{Ar^{Q+2}},$$ and then
$$|u_{(x,t),r}-u_{(x,t),R}|\leq C \left( \frac {R}{r}\right)^{\frac{Q+2}{p}}R^\alpha.$$
Let $R_i=2^{-i}R$ and $0<k<h$.  It follows that 
\begin{equation}\label{1}
|u_{(x,t),R_h}-u_{(x,t),R_k}|\leq \sum_{l=k}^{h-1}|u_{(x,t),R_{l+1}}-u_{(x,t),R_l}|\leq CR_k^\alpha ,
\end{equation}

and hence, $\left\{u_{(x,t),R_k}\right\}_{k\in \N}$ is a Cauchy sequence.  As $k\rightarrow \infty$, $u_{(x,t),R_k}\rightarrow u(x,t)$ a.e. in $\Omega$, and \eqref{1} gives
\begin{equation}\label{2}
|u(x,t)-u_{(x,t),R_k}|\leq CR_k^\alpha
\end{equation}
implying that $u$ is continuous.

In order to show that $u\in \Gamma_{loc}^\alpha (\Omega)$, we will now show that $|u(x,t)-u(x_o,t_o)|<CR^\alpha$.  Begin by choosing $(x_o,t_o)\in \Omega$ such that $0<R=d_p((x,t),(x_o,t_o))<\frac{1}{2}min\left\{R_o,diam(\Omega)\right\}.$  We need only use (\ref{2}) and estimate $|u_{(x,t),2R}-u_{(x_o,t_o),2R}|$ to obtain the result since
$$|u(x,t)-u(x_o,t_o)|\leq |u(x,t)-u_{(x,t),2R}|+|u_{(x,t),2R}-u_{(x_o,t_o),2R}|+|u(x_o,t_o)-u_{(x_o,t_o),2R}|.$$
Notice that $|u_{(x,t),2R}-u_{(x_o,t_o),2R}|$
$$\leq C \left[ \left( \int_{\Q_{2R}(x,t)\cap \Omega}\!\!\!\!\!\!\!\!\!\!\!\!\!\!\!\!\!\!\!\!\!\!\!\!\!\!\!- \,\,\,\,\,\,\,\,\,\,\,\,\,\,\,\,\,|u(z,\tau)-u_{(x,t),2R}|^p dzd\tau \right)^{1/p} + \left( \int_{\Q_{2R}(x_o,t_o)\cap \Omega}\!\!\!\!\!\!\!\!\!\!\!\!\!\!\!\!\!\!\!\!\!\!\!\!\!\!\!\!\!\!\!- \,\,\,\,\,\,\,\,\,\,\,\,\,\,\,\,\,|u(z,\tau)-u_{(x_o,t_o),2R}|^p dzd\tau \right)^{1/p} \right].$$

As before, we use the fact that $u\in C_{p,0}^{\alpha}(\Omega)$ to get that
$$|u_{(x,t),2R}-u_{(x_o,t_o),2R}|\leq CR^\alpha.$$
\end{proof}

\begin{rmrk}\label{taylor}
There is an important consequence of Theorem \ref{G} as it relates to Taylor polynomials.  Suppose $f\in C_{p,d}^{\al}(0,0)$ for $\al>d,$ then $f$ is $d$ times differentiable with H\"older continuous derivatives. (The reverse also holds.)  Let the H\"older exponent for the highest order derivative be $\al_o=\al-d$.  This implies that we can take the $d$th order Taylor expansion of $f$, and by Theorem \ref{PSTI}, $|f-P_d|\leq Cr^{d+\al_o}$.  In other words, if $\al$ is large enough (or if we know $f$ is differentiable to some order $d\leq \al$), we get that the Taylor expansion satisfies the decay requirements of the Campanato space definition with $\al$ replaced with $\al_o +d$.  This point will be exploited later to give bounds on the derivatives of solutions and inhomogenous terms.  
\end{rmrk}


\subsection{ Sobolev Spaces and Embeddings}

The following notations and estimates will be used extensively throughout this paper. In particular, the $L^p$ estimates are essential for the pointwise result.  Whereas in the classical proof as well as the group case proof of the Schauder estimates, the $L^p$ estimates are not necessary.  We start by giving the well known results for the constant coefficient case, then prove a version for the non-constant coefficient equations with only pointwise H\"older continuity assumed.  Let us first define the relevant operators. 
\begin{equation} \label{HA}
H_A=\partial_t-\sum_{i,j=1}^{m_1}a_{ij}(x,t)X_iX_j
\end{equation}
Letting $A=(a_{ij})_{i,j=1\ldots m_1}$ denote the positive definite, symmetric $m_1 \times m_1$ real-valued matrix.  Explicitly, there exists a constant, $\Lambda ,$ such that
\begin{equation}\label{ellipticity}
\Lambda^{-1}|\xi|^2\leq  \sum_{i=1}^{m_1}a_{ij}(x,t)\xi_i \xi_j  \  \leq\Lambda|\xi|^2 \ \ \mathrm{for \ any}\ \ \xi\in \mathbb{S}^{m_1-1}\subset \R^{m_1}.
\end{equation}
Also assume that $a_{ij}\in \Gamma^\al(0,0)$. 
The frozen operator is given by  
\begin{equation} \label{H0}
H_A(0)=\partial_t-\sum_{i,j=1}^{m_1}a_{ij}(0,0)X_iX_j.
\end{equation}  

Now the appropriate Sobolev spaces are defined.
\begin{dfn}
We say $f$ is in the Sobolev space, $S_p^{k,l}(\Omega)$, if and only if 
$$||f||_{S_p^{k,l}(\Omega)}:=\sum_{I=0}^k ||X^I f||_{L^p(\Omega)}+\sum_{j=0}^l ||D_t^j f||_{L^p(\Omega)}<\infty.$$
\end{dfn}

The well known Sobolev embedding theorem still holds in the parabolic setting.  This result can be found in \cite[Theorem 5.15]{Fo} for groups.  We give now a parabolic version.

\begin{thrm}\label{mce2}
Suppose $H_Au=f$ in $\Q_1$ and $f\in L^p(\Q_r).$  For $p>\frac{Q+2}{2}$, $S_p^{2,1}(\Q_r(0,0))\subset\Gamma^{\alpha}(\Q_r(0,0))$ where $\alpha=2-\frac{Q+2}{p}$.
\end{thrm}

\begin{proof}
The theorem is shown in \cite{FGN} for solutions to $H_A(0)u=g$.  They actually show that for $\alpha=2-\frac{Q+2}{p},$
$$||u||_{\Gamma^\al}(\Q_r) \leq C||g||_{L^p}(\Q_r).$$
Apply their result to
\begin{eqnarray*}
H_A(0)u&=&H_Au(0)-H_Au+H_Au\\
&=&\sum_{ij=1}^{m_1}(a_{ij}(x,t)-a_{ij}(0,0))X_iX_ju+f.
\end{eqnarray*}
Then 
$$||u||_{\Gamma^\al}(\Q_r) \leq C(r^\al||X_iX_ju||_{L^p}(\Q_r)+\lpl f\rpr),$$
which is finite for fixed $r$ under the conditions of the theorem.
\end{proof}
Now we are ready to give the interior $L^p$ estimates.
\begin{lemma}\label{Lp1}
\cite[Theorem 18]{RS} Suppose $f\in L^p_{loc}(\mathbb{G}\times \R)$ and $1<p<\infty.$ If $H_A(0)u=f$ on $\mathbb{G}\times \R$, then for any $a,b\in C^\infty_o(\mathbb{G}\times \R)$
\begin{equation}
||au||_{S_p^{2,1}(\mathbb{G}\times \R)}\leq C(||bf||_{L^p(\mathbb{G}\times \R)})
\end{equation}
for some positive constant $C=C(p,\mathbb{G},a,b).$  
\end{lemma}

Before moving to the case of variable coefficients where the coefficient functions are H\"older continuous only at a single point, we will localize the above lemma using the technique given in \cite{GT}.  This proof can also be found in \cite[section 5]{2ndlp}.

\begin{lemma}\label{local}
If $f\in L^p(\Q_{2r}(x,t))$ and $H_A(0)u=f$ on $\Q_r$, then for any $r>0$ the following result holds with $k=|I|+2l$
\begin{equation}
\sum_{k=0}^2r^{k}|| X^ID_t^lu||_{L^p(\Q_r(x,t))} \leq C(||u||_{L^p(\Q_{2r}(x,t))}+r^2||f||_{L^p(\Q_{2r}(x,t))})
\end{equation}
for some positive constant $C=C(p,\mathbb{G},r,a_{ij}(0,0)).$ 
\end{lemma}

The method of proof relies heavily on the existence of a test function with certain bounds on its derivatives (for the construction see \cite[Lemma 5]{2ndlp}) as well as an interpolation inequality \cite[Theorem 12]{2ndlp}.

  
Through the coefficient freezing technique, we obtain the following lemma for the operator with pointwise H\"older continuous coefficients.
\begin{lemma}\label{Lp2}
If $f\in L^p(\mathbb{G}\times\R)$ with compact support and $H_Au=f$ on $\mathbb{G}\times\R$, then for any $r>0$ the following result holds
\begin{equation}
\sum_{k=0}^2r^{k}|| X^ID_t^lu||_{L^p(\Q_r)}\leq C(||u||_{L^p(\Q_{2r})}+r^{2}||f||_{L^p(\Q_{2r})})
\end{equation}
for some positive constant $C=C(p,\mathbb{G},r,A).$ 
\end{lemma}

\begin{proof}
Begin by first considering $u$ with compact support   
\begin{eqnarray*}
H_A(0)u&=&H_A(0)u-H_Au+H_Au\\
&=&\sum_{ij=1}^{m_1}(a_{ij}(0,0)-a_{ij}(x,t))X_iX_ju(x,t)+f.
\end{eqnarray*}
Apply Lemma \ref{Lp1}.
\begin{equation}
||u||_{S_p^{2,1}(\mathbb{G}\times \R)}\leq C\left(\sum_{ij=1}^{m_1}\sup\limits_{\Q_r}|a_{ij}(0,0)-a_{ij}(x,t)| ||X_iX_ju||_{L^p(\mathbb{G}\times \R)}+||f||_{L^p(\mathbb{G}\times \R)}\right)
\end{equation}
Choosing $r$ small enough (say $Cr^\al<1/2$), the second derivative term gets absorbed into the left hand side giving 
\begin{equation}
||u||_{S_p^{2,1}(\mathbb{G}\times \R)}\leq C||f||_{L^p(\mathbb{G}\times \R)}
\end{equation}
For functions without compact support, apply this equation to a product of a cutoff function with a solution, $u$, and use the same localization argument as in Lemma \ref{local}.
\end{proof}

By the Sobolev embedding (Theorem \ref{mce2}), we gain as a quick corollary, that $u\in \Gamma^\al(\Q_r)$ with $\alpha=2-\frac{Q+2}{p}$, and so it is essentially bounded.
\begin{cor}\label{sup}
If $f\in L^p(\Q_{2r}(x,t))$ and $H_Au=f$ on $\Q_r(x,t)$, and additionally, if $ \frac{Q+2}{2}<p<\infty,$ then for any $r>0$
\begin{equation}
||u||_{L^{\infty}(\Q_{r}(x,t))}\leq C(||u||_{L^p(\Q_{2r}(x,t))}+||f||_{L^p(\Q_{2r}(x,t))})
\end{equation}
for some positive constant $C=C(p,\mathbb{G},r).$
\end{cor}


\section{The Pointwise Schauder Estimates}

\subsection{Preliminaries}

In this chapter, we first explore the heat kernel for the group setting as well as the existence and properties of a fundamental solution.  For easy access, the operators of interest as well as the necessary conditions on the coefficients are given here again.

\begin{equation} \label{H}
H=\partial_{t}-\sum_{i=1}^{m_1}X_{i}^{2},
\end{equation}
where $X_{1},\ldots,X_{m_1}$ generate the first layer of the Lie algebra stratification for a Carnot group, and
\begin{equation} \label{HA}
H_A=\partial_t-\sum_{i,j=1}^{m_1}a_{ij}(x,t)X_iX_j.
\end{equation}
Letting $A=(a_{ij})_{i,j=1\ldots m_1}$ denote the positive definite, symmetric $m_1 \times m_1$ real-valued matrix, the ellipticity condition states that for all $(x,t)\in \mathbb{G}\times \R$ there exists a constant, $\Lambda$ such that
\begin{equation}\label{ellipticity}
\Lambda^{-1}|\xi|^2\leq  \sum_{i=1}^{m_1}a_{ij}\xi_i \xi_j  \  \leq\Lambda|\xi|^2 \ \ \mathrm{for \ any}\ \ \xi\in \mathbb{S}^{m_1-1}\subset \R^{m_1}.
\end{equation}
We will also make use of the frozen operator, 
\begin{equation} \label{H0}
H_A(0)=\partial_t-\sum_{i,j=1}^{m_1}a_{ij}(0,0)X_iX_j.
\end{equation}
When establishing the a-priori estimates leading up to the Schauder estimates, the proofs are done for the operator $H$.  These proofs will also apply to $H_A(0)$ by making use of a linear transformation.  This will be discussed in more detail in Section 4.3.

There is a rich history behind the study of fundamental solutions. Given here are only a few results that are most relevant to the cause of this paper.  In reference \cite{GaussEst1}, the authors prove uniform Gaussian estimates on the associated heat kernel of $H$ and $H_A$.  The bounds listed below are given in \cite[Theorem 5.3]{GaussEst1}.  Building on this work, the same authors later constructed the fundamental solution, $\Gamma$, for the general operator, $H_A$, under the additional assumption that the enries of $A$ are H{\"{o}}lder continuous \cite[Theorem 1.2]{GaussEst4}. In \cite{GaussEst2}, these same results were extended from the setting of stratified groups to the case of H\"ormander vector fields by means of the Rothschild and Stein lifting and approximation theorems. 

The subsequent theorem is essentially Theorems 1.1 and 1.2 of \cite{GaussEst4}.  However, we only need the result for the constant coefficient equation, $H$.
\begin{thrm} \label{FunSol} Consider the operator $H$ given above.  Then there exists a fundamental solution $\Gamma$ for $H$ with the properties listed below.\\
(i)  $\Gamma$ is a continuous function away from the diagonal of $\R^{N+1}\times \R^{N+1}.$  Moreover, for every fixed $\zeta\in \R^{N+1}$, $\Gamma(\cdot; \zeta)\in \Gamma^{2+\alpha}_{loc}(\R^{N+1}/\{\zeta\}),$ and we have
$$H(\Gamma(\cdot; \zeta))=0\,\,\,\,in\,\,\R^{N+1}/\{\zeta\}.$$\\
(ii)  $\Gamma(x,t;y,s)=0$ for $t\leq s$.  Moreover, there exists positive constants $b$ and (for every $T>0$) a positive constant $C$ such that for $0<t-s \leq T$ the following estimates hold:
\begin{equation}\label{gam}
|\Gamma (x,t;y,s)|\leq C(t-s)^{-Q/2}exp\left(-b\left(\frac{d(x,y)^2}{(t-s)}\right)\right)
\end{equation}
and
\begin{equation}\label{dgam}
|X^{I}D_{t}^{l}\Gamma (x,t;y,s)|\leq C(t-s)^{-(Q+|I|+2l)/2}exp\left(-b\left(\frac{d(x,y)^2}{(t-s)}\right)\right)
\end{equation}
The constants $C$ and $b$ depend on $\mathbb{G}, T, A, I,$ and $l$\\
(iii) For every $f\in C^\infty_0(\R^{N+1}),$ the function
$$w(z)=\int_{\R^{N+1}}\Gamma(z;\zeta)f(\zeta)d\zeta$$
belongs to the class $\Gamma^{2+\alpha}_{loc}(\R^{N+1}),$ and we have
$$Hw=f\,\,\,\,in\,\,\R^{N+1}.$$\\
\end{thrm}
The Gaussian bounds (\ref{gam}) for $\Gamma (x,t)=\Gamma(x,t; 0,0)$ gives the following estimates for $(x,t)$ away from the singularity at the origin, which will be of great use in the next section.
\begin{eqnarray*}
|\Gamma (x,t)|&\leq& C|t|^{-Q/2}exp\left(-b\left(\frac{d(x,0)^2}{|t|}\right)\right)\\
&=&\frac{C}{|(x,t)|^Q}\left[\frac{|(x,t)|^Q}{|t|^{Q/2}}exp\left(-b\left(\frac{d(x,0)^2}{|t|}\right)\right)\right]
\end{eqnarray*}

Observing the term in brackets is bounded for $(x,t)\neq (0,0)$, we have
\begin{equation}\label{est1}
|\Gamma (x,t)|\leq\frac{C}{|(x,t)|^Q}.
\end{equation}

Similarly, for derivatives we have the estimate
\begin{eqnarray*}
|X^{I}D_{t}^{l}\Gamma (x,t)|&\leq& C|t|^{-(Q+|I|+2l)/2}exp\left(-b\left(\frac{d(x,0)^2}{|t|}\right)\right)\\
&=&\frac{C}{|(x,t)|^{Q+|I|+2l}}\left[ \frac{|(x,t)|^{Q+|I|+2l}}{|t|^{(Q+|I|+2l)/2}}exp\left(-b\left(\frac{d(x,0)^2}{|t|}\right)\right)\right],
\end{eqnarray*}
 which gives
\begin{equation}\label{est2}
|X^{I}D_{t}^{l}\Gamma (x,t)|\leq\frac{C}{|(x,t)|^{Q+|I|+2l}}.
\end{equation}

Consider now the $kth$ order Taylor polynomial (in the $x$ and $t$ variables) of $\Gamma(x,t;y,s)$ with center at the origin given by $$\Gamma_{k}(x,t;y,s)=\sum_{|I|+2l=k} c_{k}X^{I}D_{t}^{l}\Gamma((y,s)^{-1})x^{I}t^{l}.$$
Using the previous estimates, we see that 
\begin{equation}\label{est3}
|\Gamma_{k}(x,t;y,s)|\leq\sum_{|I|+2l=k} c_{k}\frac{|x|^{I}|t|^{l}}{|(y,s)|^{Q+|I|+2l}}.
\end{equation}

Making the first use of the fundamental solution, we prove two lemmas regarding $H$ applied to polynomials.  These two results appear in \cite{CH} for the sublaplacian as Lemma 3.7 and Lemma 3.8, respectively.  Here those proofs are reproduced to ensure they are still valid for the operator $H$.

\begin{lemma}\label{P1} 
If $\mathrm{Q}$ is a group polynomial of degree $d-2$, then we can always find a polynomial $P$ of degree $d$ such that $HP(x,t)=\mathrm{Q}(x,t)$ in $\mathbb{G}\times\R$.
\end{lemma}

\begin{proof}
Let $\Gamma$ be the fundamental solution as before, and define 
$$f(x,t)=\int_{\Q_1(0,0)}\Gamma((x,t)(y,s)^{-1})\mathrm{Q}(y,s)dyds$$ so that $Hf(x,t)=\mathrm{Q}(x,t)$ in $\Q_1$.  Consider the $d$th order Taylor polynomial $P_d$ of $f$ centered at the origin.  We can express $f$ as the Taylor polynomial plus some remainder term. $f=P_d+R_d$.  We would like to show that $HR_d=0$ in $\Q_1$ so that $Hf=HP_d=\mathrm{Q}$ in $\Q_1$ and thus in all of $\mathbb{G}\times\R$.

By definition, for $f\in C_o^{\infty}(\mathbb{G}\times\R)$, we have $X^ID_t^lP_d(0,0)=X^ID_t^lf(0,0)$ for all $|I|+2l\leq d$.  This implies $X^ID_t^lR_d(0,0)=0$ for all $|I|+2l\leq d$.  Additionally, by degree consideration $X^{I}D_t^lHR_d(0,0)=0$ for all $|I|+2l \leq d-2$ giving, $HX^ID_t^lR_d(0,0)=X^ID_t^lHR_d(0,0)\\=0$ for $|I|+2l\leq d-2$.  

Now consider the expansion of $\tilde{f}(x,t)=HR_d(x,t)$ as $\tilde{P}_{d-2}(x,t)+\tilde{R}_{d-2}(x,t)$.  Again, we have that 
$X^ID_t^l\tilde{f}(0,0)=X^ID_t^l\tilde{P}_{d-2}(0,0)$ for all $|I|+2l\leq d-2$.  Consequently,
$X^ID_t^l\tilde{R}_d(0,0)=0$ for all $|I|+2l\leq d-2$, and $X^ID_t^lH\tilde{R}_d=0$ for all $|I|+2l\leq d-2$.

\begin{equation}
X^ID_t^l \tilde f (0,0) =
\begin{cases}
& X^ID_t^l \tilde P_{d-2}(0,0)\,\,\,for\,\,|I|+2l\leq d-2 \\
& X^ID_t^l \tilde R_{d-2}(0,0)\,\,\,for\,\,|I|+2l> d-2 \\
\end{cases}
\end{equation}

However, by the expansion of $f$, we noted that $X^{I}D_t^lHR_d(0,0)=0$ for all $|I|+2l \leq d-2,$ so in fact,  $X^{I}D_t^l\tilde P_{d-2}(0,0)=0$ for all $|I|+2l \leq d-2.$  Then $P_{d-2}(0,0)$ has no terms with degree less than or equal to $d-2$, and $P_{d-2}$ is in fact equal to zero.

By virtue of Corollary \ref{CPSTI}, we see
\begin{eqnarray*}
|\tilde{f}-\tilde{P}_{d-2}|&\leq&c_d|(y,s)|^{d-1} \sup\limits_{\stackrel{|(z,\tau)|\leq b|(y,s)|}{|J|+2p=d-1}}|X^{J}D_{t}^{p}\tilde{f}(z,\tau)|.
\end{eqnarray*}

Making substitutions,
\begin{eqnarray}\label{a}
|\tilde{R}_{d-2}|&\leq&c_d|(y,s)|^{d-1} \sup\limits_{\stackrel{|(z,\tau)|\leq b|(y,s)|}{|J|+2p=d-1}}|X^{J}D_{t}^{p}HR_d(z,\tau)|.
\end{eqnarray}
$X^{J}D_{t}^{p}H$ is an operator of order $d+1$, so $P_d$ is annihilated when the operator is applied giving
$$X^{J}D_{t}^{p}HR_d=X^{J}D_{t}^{p}H(f-P_d)=X^{J}D_{t}^{p}Hf=X^{J}D_{t}^{p}\mathrm{Q}=0$$
because $\mathrm{Q}$ is of degree $d-2$ and $X^{J}D_{t}^{p}$ is of order $d-1$.  
Equation (\ref{a}) now reads $$|\tilde{R}_{d-2}|=|HR_d|\leq 0.$$
\end{proof}

\begin{lemma}\label{analytic}
Let $u(x,t)\in C^{\infty}(\Q_{1}(0,0))$ be a solution to $Hu=0$ in $\Q_{r}$ with $|u|\leq 1$ on $\partial \Q_1$. If we write $u(x,t)=\sum_{k=1}^d P_{k}(x,t)+R_{d}(x,t)$ where the $P_{k}$'s are the $kth$ order terms of the Taylor polynomial expansion at the origin and $R_{d}$ is the remainder term, then for every $0\leq k\leq d$, $0<r<1$, we have:
\newline
(i) $HP_k(x,t)=0$
\newline
(ii) $P_k$ has universally bounded coefficients for all $0\leq k\leq d$. 
\end{lemma}

\begin{proof}
By the argument given in the proof of Lemma \ref{P1}, we know that $HR_d=0$.  Then $Hu(x,t)=\sum_{k=1}^d HP_{k}(x,t)=0$, which implies that for each $k$, $HP_{k}(x,t)=0$.  Degree considerations rule out the possibility of any terms canceling to get zero.\\
To prove (ii), we need an interior estimate on derivatives and Bony's maximum principle proved in \cite{Bo}.  Derivative estimates (found in \cite{CGL}) give 
$$|X^ID_t^lu(x,t)|\leq Cr^{-k}\sup\limits_{\Q_r(x_o,t_o)}|u(x,t)|$$
 for all $0<r<1$ and all $|I|+2l=k$. Let $M$ be the supremum on $\Q_r(x_o,t_o)$, and the maximum principle gives $u(x,t)\leq M$ for all $(x,t)\in\mathbb{G}\times\R$ giving a bound for $|X^ID_t^lu(x,t)|$.  Consequently, the coefficients for $P_k$ are bounded.
\end{proof}


\subsection{A-priori Estimates}

\begin{lemma}\label{lemma1}  Suppose $f\in L^{p}(\Q_1(0,0))$ and $p>1+\frac{Q}{2}.$ If for some constants $\gamma>0$, $\alpha\in(0,1)$, and some integer $d\geq 2$, $f$ satisfies 
\begin{equation}\label{fdecay}
\left\|f\right\|_{L^{p}(\Q_{r})}\leq \gamma r^{d-2+\alpha+\frac{Q+2}{p}}\ \ \  for \ all \ \ r\leq1,
\end{equation} 
then there exists a function $u(x,t)\in S_{p}^{2,1}(\Q_1)$ such that $Hu=f(x,t)$ in $\Q_1$, and moreover for $k=|I|+2l,$
\begin{equation}\label{est}
\sum_{k=0}^{2}r^{k}\left\|X^{I}D_t^lu\right\|_{L^{p}(\Q_{r})}\leq C \gamma r^{d+\alpha+\frac{Q+2}{p}}\ \ \  for \ all \ \ r\leq1
\end{equation}
where $C$ is a positive constant depending on $Q, p, and d.$

\end{lemma}
\begin{proof}
The proof is in three steps.  First, we will establish the existence of a solution.  During step 2, we will establish a particular estimate on this solution, which is vital to step 3 where we obtain (\ref{est}).  

Without loss of generality, extend $f$ to equal $0$ outside of $|(x,t)|>1$.  With $\Gamma$ being the fundamental solution to $H$, we define $$w(x,t)=\int_{|(y,s)|<1}\Gamma((x,t)(y,s)^{-1})f(y,s)dyds$$ and notice that $Hw=f$ in $\Q_1$ by Theorem \ref{FunSol} . By interior estimates given in Lemma \ref{Lp1}, $\left\|w\right\|_{S_{p}^{2,1}(\Q_1)}\leq C\left\|f\right\|_{L^{p}(\Q_1)}\leq C\gamma$.

Using the $d$th order Taylor expansion of $\Gamma$ at the origin, we similarly define the function $$v(x,t)=\int_{|(y,s)|<1}\,\,\sum_{k=0}^{d}\Gamma_{k}(x,t;y,s)f(y,s)dyds.$$  
By Lemma \ref{analytic}, $H\Gamma_{k}=0$, and since the $\Gamma_k$'s are smooth and bounded, there is no issue with moving the derivatives inside the integral giving $Hv=0$ provided $v(x,t)$ is well defined.  In fact, we can show that $|v(x,t)|\leq C\gamma$ in $\Q_1$.
\begin{eqnarray*}
|v(x,t)|& \leq & \int_{|(y,s)|<1}\sum_{k=0}^{d}|\Gamma_{k}(x,t;y,s)||f(y,s)|dyds \\
& \leq & \sum_{k=0}^{d}c_{k}\int_{|(y,s)|<1}\frac{|x|^{|I|}|t|^{l}}{|(y,s)|^{Q+k}}|f(y,s)|dyds \\
& \leq & \sum_{k=0}^{d}c_{k}\sum_{i=0}^{\infty}\int_{2^{-i}\leq|(y,s)| \leq 2^{-i+1}}\frac{1}{|(y,s)|^{Q+k}}|f(y,s)|dyds \\
& \leq & C\sum_{k=0}^{d}\sum_{i=0}^{\infty}\int_{2^{-i}\leq|(y,s)|\leq2^{-i+1}}(2^i)^{Q+k}|f(y,s)|dyds .
\end{eqnarray*}
Integrating over the larger set and using H\"older's inequality $\left(\frac{1}{p}+\frac{1}{p'}=1\right)$ gives the estimation desired.
\begin{eqnarray*}
|v(x,t)|& \leq & C\sum_{k=0}^{d}\sum_{i=0}^{\infty}(2^i)^{Q+k}|\Q_{2^{-i+1}}|^{1/p'}\left\|f(y,s)\right\|_{L_{p}(|(y,s)|\leq 2^{-i+1})} \\
& \leq & C\sum_{k=0}^{d}\sum_{i=0}^{\infty}(2^i)^{Q+k}(2^{-i+1})^{(Q+2)/p'}\gamma(2^{-i+1})^{d+\alpha+\frac{Q+2}{p}-2} \\
& \leq & C\gamma \sum_{k=0}^{d}\sum_{i=0}^{\infty}(2^{-i})^{d+\alpha-k} \\
& \leq & C\gamma.
\end{eqnarray*}

Let $u=w-v$.  Then $Hu=f$.  This establishes existence of a solution.  To get $u\in S_{p}^{2,1}$ with the appropriate estimates, we will first show that $|u(x,t)|\leq C\gamma|(x,t)|^{d+\alpha}$ for $|(x,t)|<\frac{1}{D}$ where $D$ is a positive constant to be specified later.

Write $u$ as $u(x,t)=I_{1}-I_{2}+I_{3}$ where

$$I_{1}=\int_{|(y,s)|<D|(x,t)|}\Gamma((x,t)(y,s)^{-1})f(y,s)dyds,$$
$$I_{2}=\int_{|(y,s)|<D|(x,t)|}\sum_{k=0}^{d}\Gamma_{k}(x,t;y,s)f(y,s)dyds, \text{ and}$$
$$I_{3}=\int_{|(y,s)|>D|(x,t)|}\left[\Gamma((x,t)(y,s)^{-1})-\sum_{k=0}^{d}\Gamma_{k}(x,t;y,s)\right]f(y,s)dyds.$$

We will show that each of these integrals is less than $C\gamma|(x,t)|^{d+\alpha}.$ 
\begin{eqnarray*}
|I_{1}|& \leq & \int_{|(y,s)|<D|(x,t)|}|\Gamma((x,t)(y,s)^{-1})||f(y,s)|dyds \\
& \leq & \int_{|(y,s)|<D|(x,t)|}\frac{C}{|(x,t)(y,s)^{-1}|^{Q}}|f(y,s)|dyds .\\
\end{eqnarray*}
Using a change of variables and properties of the quasi-norm, $$|(z,\tau)|=|(x,t)(y,s)^{-1}|\leq A(|(y,s)|+|(x,t)|)\leq A(D|(x,t)|+|(x,t)|)\leq K|(x,t)|,$$ we get the needed estimate for a positive constant $K$ depending on the group $\mathbb{G}$ by means of a dyadic decomposition and H\"older's inequality.
\begin{eqnarray*}
|I_{1}| & \leq & \int_{|(z,\tau)|<K|(x,t)|}\frac{C}{|(z,\tau)|^{Q}}|f(y,s)|dzd\tau \\
 & \leq & \sum_{i=0}^{\infty}C\left(\int_{K2^{-i}|(x,t)|\leq|(z,\tau)|\leq K2^{-i+1}|(x,t)|}\frac{1}{|(z,\tau)|^{Qp'}}dzd\tau\right)^{1/p'}\left\|f\right\|_{L^{p}(\Q_{K2^{-i+1}|(x,t)|})} \\
&\leq& C\gamma\sum_{i=0}^{\infty}\left(\frac{2^{i}}{K|(x,t)|}\right)^{Q} \left( \frac{K|(x,t)|}{2^{i-1}}\right)^{\frac{Q+2}{p'}}\left(\frac{K|(x,t)|}{2^{i-1}}\right)^{d+\alpha-2+\frac{Q+2}{p}} \\
 & \leq & C\gamma \sum_{i=0}^{\infty}(2^{-i})^{d+\alpha}|(x,t)|^{d+\alpha} \\
 & \leq & C\gamma |(x,t)|^{d+\alpha}.
\end{eqnarray*}
The final line is obtained by noting the convergence of the geometric series.
$I_2$ is handled similarly.
\begin{eqnarray*}
|I_{2}|& \leq & \int_{|(y,s)|<D|(x,t)|}\sum_{k=0}^{d}|\Gamma_{k}(x,t;y,s)||f(y,s)|dyds \\
& \leq & \sum_{k=0}^{d}\int_{|(y,s)|<D|(x,t)|}\frac{C_{k}|x|^{|I|}|t|^{l}}{|(y,s)|^{Q+k}}|f(y,s)|dyds. \\
\end{eqnarray*}
Noticing that by binomial expansion $|(x,t)|^{k} \geq |x|^{|I|}|t|^{l},$ where $|I|+2l=k$ gives,
\begin{eqnarray*}
|I_{2}| & \leq & C\sum_{k=0}^{d} \int_{|(y,s)|<D|(x,t)|}\frac{|(x,t)|^{k}}{|(y,s)|^{Q+k}}|f(y,s)|dyds \\
 & \leq & C\sum_{k=0}^{d}\sum_{i=0}^{\infty}\left(\int_{D2^{-i}|(x,t)|\leq|(y,s)|\leq D2^{-i+1}|(x,t)|}\frac{|(x,t)|^{k}}{|(y,s)|^{(Q+k)p'}}dyds\right)^{1/p'} \left\|f\right\|_{L^{p}(\Q_{D2^{-i+1}|(x,t)|})} \\
 & \leq & C\gamma\sum_{k=0}^{d}|(x,t)|^{k}\sum_{i=0}^{\infty}\left(\frac{2^{i}}{D|(x,t)|}\right)^{Q+k} \left(\frac{D|(x,t)|}{2^{i-1}}\right)^{\frac{Q+2}{p'}}\left(\frac{D|(x,t)|}{2^{i-1}}\right)^{d+\alpha-2+\frac{Q+2}{p}} \\
 & \leq & C\gamma \sum_{i=0}^{\infty}(2^{-i})^{d+\alpha}|(x,t)|^{d+\alpha} \\
 & \leq & C\gamma |(x,t)|^{d+\alpha}.
\end{eqnarray*}

$$|I_{3}| \leq \int_{|(y,s)|\geq D|(x,t)|} |\Gamma(x,t;y,s)-\sum_{k=0}^{d}\Gamma_{k}(x,t;y,s)||f(y,s)|dyds$$
Using Corollary \ref{CPSTI},
\begin{eqnarray*}
|I_{3}|& \leq & \int_{|(y,s)|\geq D|(x,t)|}C|(x,t)|^{d+1} \sup\limits_{\stackrel{z\in \Q(0,b|(x,t)|)}{|I|+2l=d+1}} |X^{I}D_{t}^{l}\Gamma(y,s;z,\tau)||f(y,s)|dyds \\
& \leq & \int_{|(y,s)|\geq D|(x,t)|}C|(x,t)|^{d+1} \sup\limits_{z\in \Q(0,b|(x,t)|)} |(z,\tau)^{-1}(y,s)|^{-d-1-Q}|f(y,s)|dyds
\end{eqnarray*}

Since $|(z,\tau)|\leq b|(x,t)|$ and $|(y,s)|\geq|(x,t)|,$ we have by the reverse quasi-triangle inequality,
$$|(z,\tau)^{-1}(y,s)|\geq (1/A)|(y,s)|-|(z,\tau)|\geq (D/A)|(x,t)|-b|(x,t)|)\geq C'|(x,t)|.$$  Now we can see that choosing $D>Ab$ (whose dependence is determined by $\mathbb{G}$ and $d$), will give a positive constant $C'$ above.  
\begin{eqnarray*}
|I_{3}|& \leq & C(|(x,t)|^{d+1}\int_{|(y,s)|\geq D|(x,t)|}|(x,t)|^{-d-1-Q} |f(y,s)|dyds \\
& \leq & C|(x,t)|^{-Q}\sum_{i=0}^{J}\left(\int_{2^{-i}D|(x,t)|\leq|(y,s)|\leq 2^{-i+1}D|(x,t)|} 1 dyds\right)^{1/p'}\left\|f\right\|_{L^{p}(\Q_{2^{-i+1}|(x,t)|})} \\
& \leq & C\gamma |(x,t)|^{d+\alpha} \sum_{i=0}^{J}(2^{-i})^{d+\alpha+Q}\\
& \leq & C\gamma |(x,t)|^{d+\alpha}
\end{eqnarray*}

In the previous analysis, $J\in \N$ such that $2^{J-1}D|(x,t)|\leq 1 \leq 2^{J}D|(x,t)|.$  Using properties of logs, one can see that $J=[-log_2(D|(x,t)|)]$.

This completes step 2 of the proof.  Before establishing the final conclusion, notice that because we have $|u(x,t)|\leq C\gamma |(x,t)|^{d+\alpha}$ in $\Q_1,$ we obtain $\left\|u\right\|_{L^{p}(\Q_{r}(0))}\leq C\gamma |(x,t)|^{d+\alpha+\frac{Q+2}{p}}$ by once again making use of dyadic decomposition.

Let $\widetilde{u}(x,t)=u(\delta_{r}x,r^{2}t)$.  Applying the operator to $\widetilde{u}(x,t)$ gives, $H\widetilde{u}(x,t)=r^{2}Hu(x,t)=r^{2}f(x,t)$ on $\Q_{r}$.  Now apply interior estimates (Lemma \ref{local}) to $\widetilde{u}(x,t)$ to get our final conclusion.
\begin{eqnarray*}
\left\|\widetilde{u}(x,t)\right\|_{S_{p}^{2,1}(\Q_{\frac{r}{2}}(0))}&\leq&C(\left\|\widetilde{u}(x,t)\right\|_{L^{p}(\Q_{r}(0))}+\left\|H\widetilde{u}(x,t)\right\|_{L^{p}(\Q_{r}(0))}) \ \ for \ all\ r\leq \frac{1}{2}\\
&\leq&C(\left\|u(x,t)\right\|_{L^{p}(\Q_{1}(0))}+r^{2}\left\|f(x,t)\right\|_{L^{p}(\Q_{r}(0))}) \ \ for \  all\ r\leq \frac{1}{2} \\
&\leq&C\gamma r^{d+\alpha+\frac{Q+2}{p}}.
\end{eqnarray*}
The chain rule and a scaling argument completes the proof giving $$\sum_{k=0}^{2} r^k\left\|X^{I}D_t^lu(x,t)\right\|_{L^{p}(\Q_{r})}\leq C \gamma r^{d+\alpha+\frac{Q+2}{p}}\ \ for \ all\ r \leq 1.$$
\end{proof}

\begin{cor}\label{cor1}
Suppose $f\in L^{p}(\Q_1)$, $p>1+\frac{Q}{2}$, satisfies
$$\left\|f\right\|_{L^{p}(\Q_r)}\leq \gamma r^{d-2+\alpha+\frac{Q+2}{p}}\ \ for \ all \ r\leq 1$$
for some positive constants $\gamma>0$, $\alpha\in(0,1)$, and some integer $d\geq 2.$ For any solution $u(x,t)\in S_{p}^{2,1}(\Q_1)$ to $Hu=f,$ there exists a polynomial $P_d$ of degree no greater than $d$ with $HP_{d}=0$ such that
$$|u(x,t)-P_{d}(x,t)|\leq C(\gamma + \left\|u\right\|_{L^{p}(\Q_{1})})|(x,t)|^{d+\alpha} \ \ for \ all \ (x,t)\in \Q_{\half}$$
where $C>0$ is a constant depending only on $Q,\mathbb{G},p,d,\lambda,$ and $\alpha.$
\end{cor}

\begin{proof}
By Lemma \ref{lemma1}, there exists $v\in S_{p}^{2,1}(\Q_1)$ with $Hv=f$ such that 
$$|v(x,t)|\leq C\gamma|(x,t)|^{d+\alpha} \ \ for \ all \ (x,t)\in \Q_{1/2}$$
and
$$\left\|v(x,t)\right\|_{L^{p}(\Q_{1})}\leq C\gamma .$$

Note that $H(u-v)=0$, by Lemma \ref{analytic} we can write the $dth$ order Taylor expansion of $u-v$ as
$u-v=P_{d}+R_{d}.$  Moreover, using Corollary \ref{CPSTI}, we have a bound on $R_{d}$.
\begin{eqnarray*}
|R_{d}|& \leq & C|(x,t)|^{d+1} \sup\limits_{\stackrel{|(x,t)|\leq\half}{|I|+2l=d+1}} |X^{I}D_t^l(u-v)| \\
& \leq & C|(x,t)|^{d+1} \sup\limits_{|(x,t)|\leq\frac{3}{4}} |(u-v)| \\
& \leq & C|(x,t)|^{d+1} \left\|u-v\right\|_{L^{p}(\Q_{3/4})}\,\,\,\,\text{by Corollary}\,\,\ref{sup} \\ 
& \leq & C(\gamma + \left\|u\right\|_{L^{p}(Q_1)})|(x,t)|^{d+1}.
\end{eqnarray*}
The conclusion is reached by virtue of the above estimate and the previously mentioned use of Lemma \ref{lemma1}.
$$|u-P_{d}|=|v+R_d| \leq |v|+|R_{d}|\leq C_1\gamma |(x,t)|^{d+\al} +C_2(\gamma + \left\|u\right\|_{L^{p}(Q_1)})|(x,t)|^{d+1}.$$  
Since $\al<1$ and $|(x,t)|<1$, the smaller exponent gives the larger bound, and we reach the conclusion,

$$|u-P_{d}| \leq C(\gamma + \left\|u\right\|_{L^{p}(\Q_1)})|(x,t)|^{d+\alpha}.$$
\end{proof}


\subsection{Pointwise Schauder Estimates}

We now turn our attention to equations of the following form:
\begin{equation}\label{L}
H_Au(x,t)=\partial_{t}u(x,t)-\sum_{i,j=1}^{m_1}a_{ij}(x,t)X_i X_j u(x,t)=f(x,t)
\end{equation}
where the matrix $A=(a_{ij})$ satisfies for some $\alpha\in(0,1)$ and $1<\lambda\leq \Lambda< \infty$
\begin{equation}\label{hypot1}
a_{ij}\in\Gamma^{\alpha}(0,0) 
\end{equation} and
\begin{equation}\label{hypot2}
\lambda|\xi|^2\leq  \  \sum_{i,j=1}^{m_1} a_{ij}(x,t) \xi_i \xi_j  \  \leq\Lambda|\xi|^2 \ \ \mathrm{for \ any}\ \ \xi\in \R^{m_1}.
\end{equation}

Since $A=(a_{ij}(0,0))$ is positive definite, we can find a matrix $B\in GL(m_1)$ such that $BB^{T}=A$.  The vector fields $\widetilde{X}_{i}=\sum_{j=1}^{m_1}b_{ij}X_{j}$ for $i=1,\ldots,m_1$ and all their commutators will still generate the complete Lie algebra, and consequently, Lemma \ref{lemma1} and Corollary \ref{cor1} from the previous section will still hold for solutions to the frozen operator,
$$H_A(0)u= \partial_{t}u(x,t)-\sum_{i=1}^{m_1}\widetilde{X}_{i}^{2}u(x,t)=\partial_{t}u(x,t)-\sum_{i,j=1}^{m_1}a_{ij}(0,0)X_iX_ju(x,t)=f(x,t).$$

\begin{thrm}\label{big}
Let $u\in S_{p}^{2,1}(\Q_1)$ be a solution to $H_Au=f$ in $\Q_1$, with $f\in L^{p}(\Q_1)$ and $Q/2 +1<p<\infty.$  Assume the following:
\newline
1. There exists a homogeneous polynomial $\mathrm{Q}$ of degree $d-2$ such that
\begin{equation}\label{h1}
\left\|f-\mathrm{Q}\right\|_{L^{p}(\Q_{r})}\leq \gamma r^{d-2+\alpha+\frac{Q+2}{p}}.
\end{equation}
2. There exists a constant $\beta\in(0,1]$ such that
\begin{equation}\label{H}
\limsup_{r\rightarrow 0}\frac{\left\|u\right\|_{L^{p}(\Q_r)}}{r^{d-1+\beta+\frac{Q+2}{p}}}<\infty.
\end{equation}

Then there exists a constant $C=C(G,A)$ such that 
\begin{equation}\label{cu}
\left\|u\right\|_{L^{p}(\Q _{r})}\leq C\left(\left\|u\right\|_{L^{p}(\Q_{1})}+\left\|\mathrm{Q}\right\|_{L^{p}(\Q _1)}+\gamma\right) r^{d+\frac{Q+2}{p}}
\end{equation}
for any $0<r\leq1.$  Moreover, there exists a homogeneous polynomial $P$ of degree $d$ such that $H_A(0)P=\mathrm{Q}$, 
\begin{equation}\label{c1}
|P(x,t)|\leq C\left(\left\|u\right\|_{L^{p}(\Q_{1})}+\left\|\mathrm{Q}\right\|_{L^{p}(\Q _1)}+\gamma\right)|(x,t)|^{d},
\end{equation}
and
\begin{equation}\label{c2}
|u-P|(x,t)\leq C\left(\left\|u\right\|_{L^{p}(\Q_{1})}+\left\|\mathrm{Q}\right\|_{L^{p}(\Q _1)}+\gamma\right)|(x,t)|^{d+\alpha}
\end{equation}
for any $r\leq R.$ ($R$ is a constant to be fixed during the proof.)

Furthermore, for $j=|I|+2l,$
\begin{equation}\label{c3}
\sum_{j=0}^{2}r^{j}\left\|X^{I}D_t^l(u-P)\right\|_{L^{p}(\Q_{r})}\leq C\left(\left\|u\right\|_{L^{p}(\Q_{1})}+\left\|\mathrm{Q}\right\|_{L^{p}(\Q_{1})}+\gamma\right)r^{d+\alpha+\frac{Q+2}{p}}.
\end{equation}
\end{thrm}

\begin{proof}
We will divide this proof into two steps.  In the first step, we will show for a fixed $0<\alpha_1 \leq \alpha$, 
\begin{equation}
C_{k}:=\sup_{0<r\leq1}\frac{\lpl u\rpr}{r^{d-1+\beta+k\alpha_{1}+\frac{Q+2}{p}}}<\infty,
\end{equation}
provided $\beta+k\alpha_1 \leq1$.  Assumption 2 establishes the case $k=0,$ which gives
$$\lpl u\rpr\leq C_or^{d-1+\beta+\frac{Q+2}{p}}.$$
The proof will proceed by showing the case $k=1$ follows from this assumption.  We also assume $\alpha+\beta<1$.  During the induction argument, the construction of the polynomial $P$ will begin, but additional arguments to complete the estimates on $P$ and $u-P$ will be needed.  That will be the second step of the proof.

We begin step one by recalling the $L^p$ estimates for the non-constant coefficient equation in Lemma \ref{Lp2}.  For any $0<r<1/2$,
\begin{equation}\label{e1}
\begin{split}
\sum_{j=0}^{2}r^{j}\lpl X^{I}D_t^lu\rpr & \leq C(\lpl u\rprr +r^{2}\lpl f\rprr)\\
& \leq C(C_o r^{d-1+\beta+\frac{Q+2}{p}}+r^{2}(\gamma r^{d-2+\alpha+\frac{Q+2}{p}}+\lpl Q\rprr))\\
& \leq C(C_o r^{d-1+\beta+\frac{Q+2}{p}}+\gamma r^{d+\alpha+\frac{Q+2}{p}}+r^{d+\frac{Q+2}{p}}\lpl Q\rpo)\\
& \leq C(C_o+\gamma + \lpl Q\rpo)r^{d-1+\beta+\frac{Q+2}{p}}.
\end{split}
\end{equation}

Taking  $H_A(0)=\partial_t-\sum_{i,j=1}^{m_1}a_{ij}(0,0)X_iX_j,$ we can use Lemma \ref{P1} to get a polynomial of homogeneous degree $d$ such that $H_A(0)P_{1}=Q$.  Since $Q$ is assumed to be homogeneous, we can choose $P_1$ to be homogeneous as well.  
Observe that if we write 
$H_A(u-P_{1})=\tilde{\phi}$, we get an estimate on $\tilde{\phi};$
\begin{eqnarray*}
\tilde{\phi}&=& f-H_AP_1+H_A(0)P_1-H_A(0)P_1\\
&=& f-Q + (H_A(0)-H_A)(P_1) \\
&=& f-Q + \sum_{i,j=1}^{m_1}(a_{ij}-a_{ij}(0,0))X_i X_jP_1 .
\end{eqnarray*}

Using the H{\"{o}}lder continuity and the $L^p$ estimates on $H_A(0)P_1=Q$, we have 
\begin{equation}
\begin{split}
\sum_{i,j=1}^{m_1}\lpl (a_{ij}-a_{ij}(0,0))X_iX_j P_1 \rpr &\leq Cr^\al \lpl X_iX_jP_1\rpr\\
&\leq Cr^\al \left( r^{-2}\lpl P_1 \rprr +\lpl Q\rprr \right)\\
&\leq C \left(\lpl P_1 \rpo +\lpl Q\rpo \right) r^{d-2+\alpha+\frac{Q+2}{p}}.
\end{split}
\end{equation}
The factors of $r^{d}, r^{d-2}$ come from the degrees of homogeneity of $P_1$ and $Q$ respectively. 

Now the estimation of $\tilde{\phi}$ can follow.
\begin{eqnarray}
\lpl \tilde{\phi}\rpr & \leq & \lpl f-Q \rpr+ \sum_{i,j=1}^{m_1}\lpl (a_{ij}-a_{ij}(0,0))X_iX_j P_1\rpr \nonumber \\
&\leq&  \gamma r^{d-2+\alpha+\frac{Q+2}{p}} + C \left( \lpl P_1 \rpo +\lpl Q\rprr \right) r^{d-2+\alpha+\frac{Q+2}{p}} \nonumber \\
&\leq& C\left(\gamma +\lpl P_1 \rpo +\lpl Q \rpo \right) r^{d-2+\alpha+\frac{Q+2}{p}}.\label{e3}
\end{eqnarray}

It is clear that \eqref{e3} also holds for any $q$ replacing $p$ where $1<q<p.$
Now apply $L^p$ estimates to $H_A(u-P_1)=\tilde \phi$.
\begin{eqnarray}
\sum_{j=0}^{2}r^{j}\lpl X^{I}D_t^l(u-P_1)\rpr\!\!\! &\leq& \!\!\! C\left( \lpl (u-P_1)\rprr +r^2 \lpl \tilde \phi \rprr \right) \nonumber \\
&\leq& \!\!C\lpl u\rprr +C'\left( \lpl P_1 \rpo + \gamma +\lpl Q \rpo \right)r^{d-2+\alpha+\frac{Q+2}{p}}\nonumber \\
& \leq & \!\!C_o r^{d-1+\beta+\frac{Q+2}{p}} +C' \left( \lpl P_1 \rpo + \gamma +\lpl Q \rpo \right) r^{d-2+\alpha+\frac{Q+2}{p}}  \nonumber \\
& \leq & \!\!C (C_o+ \gamma+\lpl Q\rpo +\lpl P_1 \rpo )r^{d-1+\beta+\frac{Q+2}{p}}. \label{e4}
\end{eqnarray}

Now define $F=H_A(0)(u-P_{1})$, so $F=H_A(0)(u-P_{1})-H_A(u-P_{1})+H_A(u-P_{1}).$  Explicitly,  
$$F=\sum_{i,j=1}^{m_1}(a_{ij}(x,t)-a_{ij}(0,0))X_iX_j(u-P_1)+\tilde{\phi},$$ and we can apply \eqref{e3} and \eqref{e4} to get an estimate on $F$ that satisfies the hypothesis of Corollary \ref{cor1}.  The first term will be dealt with in such a way that we can see from where $\alpha_1$ comes.

For any $q$ with $1+\frac{Q}{2}<q<p$ we have
\begin{equation*}
\begin{split}
&\sum_{i,j=1}^{m_1}||(a_{ij}(x,t)-a_{ij}(0,0))X_iX_j(u-P_1)||_{L^q(\Q_r)}\\
&\leq \sum_{i,j=1}^{m_1}||(a_{ij}(x,t)-a_{ij}(0,0))||_{L^{\frac{pq}{p-q}}(\Q_r)} \lpl X_iX_j(u-P_1)\rpr\\
&\leq C\sum_{i,j=1}^{m_1}||(a_{ij}(x,t)-a_{ij}(0,0))||_{L^{\frac{pq}{p-q}}(\Q_r)}\cdot
(C_{o}+\gamma+\lpl Q\rpo +\lpl P_1 \rpo)r^{d-3+\beta+\frac{Q+2}{p}}.
\end{split}
\end{equation*}

Now, we handle two cases.  If $p>2(1+\frac{Q}{2})$, take $q=p/2>1+\frac{Q}{2}$ so that $\frac{pq}{p-q}=p$.  Therefore,
\begin{equation}
\begin{split}
\sum_{i,j=1}^{m_1}&||(a_{ij}(x,t)-a_{ij}(0,0))X_iX_j(u-P_1)||_{L^q(\Q_r)}\\
&\leq C(C_{o}+\gamma+\lpl Q\rpo +\lpl P_1 \rpo)r^{d-3+\beta+\alpha+\frac{Q+2}{p}}.
\end{split}
\end{equation}

If $p\leq 2(1+\frac{Q}{2})$, take any $q$ where $1+\frac{Q}{2}<q<p$ so that $\frac{pq}{p-q}>p$. Then,
\begin{equation*}
\begin{split}
\sum_{i,j=1}^{m_1}||(a_{ij}(x,t)-a_{ij}(0,0))||_{L^{\frac{pq}{p-q}}(\Q_r)} &\leq C \sum_{i,j=1}^{m_1}||(a_{ij}(x,t)-a_{ij}(0,0))||_{L^{p}(\Q_r)}^{\frac{p-q}{q}} \\
&\leq Cr^{\left(\alpha+\frac{Q+2}{p}\right)\cdot\left({\frac{p-q}{q}}\right)}.
\end{split}
\end{equation*}

Hence,
\begin{equation*}
\begin{split}
\sum_{i,j=1}^{m_1}||(a_{ij}(x,t)-a_{ij}(0,0))&X_iX_j(u-P_1)||_{L^q(\Q_r)}\\
&\leq  C(C_{o}+\gamma+\lpl Q\rpo +\lpl P_1 \rpo)r^{d-3+\beta+\frac{\alpha(p-q)}{q}+\frac{Q+2}{p}}.
\end{split}
\end{equation*}

In either case, we have for any $r\leq 1/2$,
\begin{equation}\label{e5}
\begin{split}
\sum_{i,j=1}^{m_1}&||(a_{ij}(x,t)-a_{ij}(0,0))X_iX_j(u-P_1)||_{L^q(\Q_r)}\\
&\leq C(C_{o}+\gamma+\lpl Q\rpo +\lpl P_1 \rpo)r^{d-3+\beta+\alpha_1+\frac{Q+2}{q}}
\end{split}
\end{equation}
for some $\alpha_1=\frac{\al(p-q)}{q} \leq \alpha$ and some $q$ with $1+\frac{Q}{2}<q<p.$ 

And finally, 
\begin{equation}\label{F}
\lpl F \rpr \leq C(C_{o}+\gamma+\lpl Q\rpo +\lpl P_1 \rpo)r^{d-3+\beta+\alpha_1+\frac{Q+2}{p}}.
\end{equation}

Applying Corollary \ref{cor1} with $\alpha_1+\beta$ replacing $\alpha$ (which is acceptable since $\alpha_1 +\beta <1$) and $d-1$ replacing $d$, we obtain a polynomial $P_{o}$ of degree $d-1$ such that
\begin{equation}
|u-P_1-P_o|(x,t)\leq C(C_{o}+\gamma+\lpl Q\rpo +\lpl P_1 \rpo +\lpl u \rpo )|(x,t)|^{d-1+\alpha_1+\beta}.
\end{equation}

Now we are in a position to see that $P_o$ is in fact identically $0$ through the following argument.
\begin{eqnarray}
|u-P_o|&=&|u-P_o-P_1+P_1| \nonumber \\
&\leq&|u-P_o-P_1|+|P_1| \nonumber \\
&\leq& C(C_{o}+\gamma+\lpl Q\rpo +\lpl P_1 \rpo +\lpl u \rpo )|(x,t)|^{d-1+\alpha_1+\beta} +\tilde C|(x,t)|^{d} \nonumber \\
&\leq& C(C_{o}+\gamma+\lpl Q\rpo +\lpl P_1 \rpo +\lpl u \rpo )|(x,t)|^{d-1+\alpha_1+\beta} .
\end{eqnarray}
Thus,
\begin{equation}\label{e6}
\lpl u-P_o \rpr \leq C(C_{o}+\gamma+\lpl Q\rpo +\lpl P_1 \rpo +\lpl u \rpo )|(x,t)|^{d-1+\alpha_1+\beta +\frac{Q+2}{p}}.
\end{equation}

And since $$\lpl P_o \rpr = \lpl P_o -u+u\rpr ,$$
\begin {eqnarray}
\lpl P_o \rpr &\leq& \lpl u-P_o\rpr +\lpl u\rpr \nonumber \\
&\leq& C(C_{o}+\gamma+\lpl Q\rpo +\lpl P_1 \rpo +\lpl u \rpo )r^{d-1+\alpha_1+\beta +\frac{Q+2}{p}} \nonumber\\
&& +C_o r^{d-1+\beta+\frac{Q+2}{p}} \nonumber \\
&\leq& Cr^{d-1+\beta+\frac{Q+2}{p}}. \nonumber \\
\end{eqnarray}

Thus, $\lpl P_o\rpo \leq Cr^{\beta}$ which implies $P_o\equiv 0$.  Looking back to equation \eqref{e6}, we see that
\begin {eqnarray*}
\lpl u\rpr &\leq& C(C_{o}+\gamma+\lpl Q\rpo +\lpl P_1 \rpo +\lpl u \rpo )|(x,t)|^{d-1+\alpha_1+\beta +\frac{Q+2}{p}}
\end{eqnarray*}

Finally, this gives $C_1 < \infty$.  Repeating the argument $k$ times we obtain
\begin{equation}
\sup_{0<r\leq1} \frac {\lpl u \rpr}{r^{d-1+\beta+k\alpha_1+\frac{Q+2}{p}}}< \infty,
\end{equation}
as long as $\beta +k\alpha < 1$.  By induction, the $C_k$'s are finite for all $k$, and we can use this fact to complete the construction of a polynomial $P$ with the bounds in the theorem.  First notice that  in the last step of iteration (from $k$ to $k+1$) there is a gain on the degree of the polynomial obtained from $d-1$ to $d$.  To see this this, first notice that for some $0<\alpha_o<1$, $1+\alpha_o<\beta +(k+1)\alpha .$  For this value of $k$, we revisit \eqref{F} and see
\begin{equation}
\lpl F \rpr \leq C(C_{k}+\gamma+\lpl Q\rpo +\lpl P_1 \rpo)r^{d-3+(1+\alpha_{o})+\frac{Q+2}{p}}.
\end{equation}

Applying Corollary \ref{cor1} again gives a polynomial $P_2$ of degree $d$ such that $H_A(0)P_2 =0$ and 
\begin{equation}\label{u-p}
|u-P_1 - P_2 |(x,t)\leq \left[C(C_k +\gamma +\lpl Q\rpo +\lpl P_1 \rpo)+\lpl u \rpo \right]|(x,t)|^{d+\alpha_o}
\end{equation}
for all $(x,t)\in \Q_r.$  Set $P=P_1 +P_2$, so $P$ is a homogeneous polynomial of degree $d$ and $H_A(0) P=Q$.
\begin{rmrk}
The degree of $P$ is clear, but to see that $P$ is in fact homogeneous, we first write $P$ as the sum of homogeneous polynomials and discover that only the homogeneous part with degree $d$ is nonzero. Let $P=\sum_{j=0}^{d}\tilde P_{j}$, and recall \eqref{H} and \eqref{u-p}, which state
\begin{equation}
|u-P|(x,t) \leq C|(x,t)|^{d+\alpha_o}
\end{equation} and
\begin{equation}
\lpl u \rpr \leq Cr^{d-1+\beta+\frac{Q+2}{p}}.
\end{equation}
Then 
\begin{eqnarray*}
\lpl \sum_{j=0}^{d} \tilde P_j \rpr= \lpl P\rpr &\leq& \lpl u-P\rpr + \lpl u\rpr  \\
&\leq& (\int_{\Q_r}|(x,t)|^{p(d+\alpha_o)}dxdt)^{1/p} + Cr^{d-1+\beta+\frac{Q+2}{p}} \\
&\leq& Cr^{d+\alpha_o +\frac{Q+2}{p}} + Cr^{d-1+\beta+\frac{Q+2}{p}} \\
&\leq& Cr^{d-1+\beta+\frac{Q+2}{p}}.
\end{eqnarray*}
This implies $\lpl \tilde P_{j} \rpr \leq Cr^{d-1+\beta+\frac{Q+2}{p}}$ for every $j=0,1,\ldots,d$.  By homogeneity and a change of variable, we obtain
$$r^{j+\frac{Q+2}{p}}\lpl \tilde P_{j} \rpo \leq Cr^{d-1+\beta+\frac{Q+2}{p}}$$
and
$$\lpl \tilde P_{j} \rpo \leq Cr^{d-j+\beta-1}.$$
The last estimate ensures $\tilde P_{j}=0$ unless $j=d$.
\end{rmrk}

This completes step 1 of the proof.  For the second step, we wish to remove the dependence on $P_1$ and $C_k$ in the constants bounding $u$ and $P$.  To this end, we prove estimates under the additional assumption that $0<R<1$ is small enough that
$$\sup_{\Q_{R}}|a_{ij}(x,t)-a_{ij}(0,0)| \leq \eta < 1/2 $$ for some small $\eta>0$ to be chosen later in the proof.  This assumption can be made without loss of generality.   The general case can be recovered by applying the transformation $(x,t)\rightarrow(Rx,R^2t)$ for a suitable $R\in(0,1).$
Let $\psi=u-P$ and 
$$\delta=\sup_{0<r<R} \frac {\lpl \psi \rpr}{r^{d+\alpha+\frac{Q+2}{p}}}.$$
From step 1 (\ref{u-p}), we know $\delta$ is finite.
\begin{eqnarray*}
H_A\psi &=& H_A(u-P) \\
&=&H_Au-H_A(0)P+H_A(0)P-H_AP \\
&=&f-Q+\sum_{i,j=1}^{m_1}(a_{ij}(x,t)-a_{ij}(0,0))X_iX_jP.
\end{eqnarray*}

We apply $L^p$ estimates again and get for any $r\leq R$,
\begin{equation*}
\begin{split}
\sum_{j=0}^2 r^{j}\lpl X^ID_t^l\psi \rpr &\leq C(\lpl \psi \rprr + r^2\lpl f-Q\rprr +r^2\lpl (H_A(0)-H_A)P\rprr)\\
&\leq C(\delta+\gamma+\lpl P\rpo)r^{d+\alpha+\frac{Q+2}{p}}
\end{split}
\end{equation*}
Consider $H_A(0)\psi=\tilde{F}$ written as $\tilde{F}=H_A(0)\psi-H_A\psi+H_A\psi$ or alternatively, 
$$\tilde{F}=\sum_{i,j=1}^{m_1}(a_{ij}(x,t)-a_{ij}(0,0))X_iX_j\psi +H_A\psi.$$ Then we see that 
\begin{eqnarray}
\lpl \tilde{F}\rpr &\leq& C[\eta(\delta +\gamma + \lpl P\rpo)r^{d-2+\alpha+\frac{Q+2}{p}} +(\gamma +\lpl P\rpo)r^{d-2+\alpha+\frac{Q+2}{p}}] \nonumber \\
&\leq& C[\eta(\delta +\gamma + \lpl P\rpo) +(\gamma +\lpl P\rpo)]r^{d-2+\alpha+\frac{Q+2}{p}}. \label{tilf}
\end{eqnarray}

Using Corollary \ref{cor1}, there exists a polynomial $P_3$ of degree $d$ such that
\begin{equation}
\lpl \psi-P_3 \rpr\leq C\left(\eta\delta+\eta (\gamma+\lpl P \rpo)+\gamma+\lpl P\rpo +\lpl \psi \rpo\right) r^{d+\alpha+\frac{Q+2}{p}}.
\end{equation}
By the same argument given for $P_o$ in step 1, $P_3\equiv 0$, and 
\begin{equation*}
\delta\leq C\left(\eta\delta+\eta (\gamma+\lpl P \rpo)+\gamma+\lpl P\rpo +\lpl \psi \rpo\right).
\end{equation*}
Choose $\eta<1/C$ so that $1-C\eta$ is positive and 
\begin{equation*}
(1-C\eta)\delta\leq C\left((\eta+1) (\gamma+\lpl P \rpo)+\gamma+\lpl P\rpo +\lpl \psi \rpo\right).
\end{equation*}

By designating a new constant $C'$ whose dependence is the same as the old constant $C$, we have
\begin{equation}\label{delta}
\delta\leq C'(\gamma+\lpl P \rpo+\lpl \psi \rpo).
\end{equation}

Equivalently from the definition of $\delta$, for $|(x,t)|<R$  
\begin{eqnarray}
\lpl \psi \rpr &\leq& C(\gamma+\lpl P\rpo +\lpl \psi \rpo)r^{d+\alpha+\frac{Q+2}{p}}\nonumber\\
&\leq& C(\gamma+\lpl u\rpo +\lpl P \rpo)r^{d+\alpha+\frac{Q+2}{p}},
\end{eqnarray}
and by (\ref{tilf})
\begin{equation}
\lpl \tilde{F}\rpr \leq C(\gamma+\lpl u\rpo + \lpl P\rpo)r^{d-2+\alpha+\frac{Q+2}{p}}.
\end{equation}
This allows us to once again make use of Corollary \ref{cor1} and get $\tilde P$ of degree $d$ such that
\begin{equation}
|\psi-\tilde P|\leq C(\gamma+\lpl P\rpo + \lpl u \rpo)|(x,t)|^{d+\alpha} \,\,\, \mathrm{in}\,\,\Q_{R}.
\end{equation}
But by the same argument used before, this new polynomial is zero as well, and
\begin{equation}\label{psi}
|\psi|\leq C(\gamma+\lpl P\rpo + \lpl u \rpo)|(x,t)|^{d+\alpha} \,\,\, \mathrm{in}\,\,\Q_{R}.
\end{equation}
We are now in a position to establish the estimates for $P$ and $u-P$.
By the definition of $\psi$ and using the fact that $\lpl P\rpo \leq C||P||_{L^\infty(\Q_1)}$, we see
\begin{eqnarray*}
|P(x,t)|&\leq& |u(x,t)|+|(u-P)(x,t)|\\
&\leq&|u(x,t)|+C(\gamma+\lpl u\rpo + ||P||_{L^\infty(\Q_1)})|(x,t)|^{d+\alpha}.
\end{eqnarray*}
And interior estimates (Corollary \ref{sup}) imply that
\begin{equation}\label{u}
|u(x,t)|\leq C(\gamma+\lpl Q\rpo +\lpl u \rpo)\,\,\,\mathrm{in}\,\,\,\Q_{R}.
\end{equation}

The implication follows as
\begin{eqnarray*}
|P(x,t)|&\leq& C(\gamma+\lpl Q\rpo +\lpl u \rpo)+C'(\gamma+\lpl u\rpo + ||P||_{L^\infty(\Q_1)})|(x,t)|^{d+\alpha}\\
&\leq& C(\gamma+\lpl Q\rpo +\lpl u \rpo)+C'' ||P||_{L^\infty(\Q_1)}|(x,t)|^{d+\alpha} \,\,\,\mathrm{in}\,\,\, \Q_R.
\end{eqnarray*}
Suppose $P$ restricted in $\{(e_x,e_t)\in \mathbb{G}\times \R; |(e_x, e_t)|=1\}$ attains its maximum at $(z,\tau)$.  Choose $x=|(x,t)|z$ and $t=|(x,t)|^2\tau$.  By the homogeneity of $P$, 
$$|P(x,t)|=|(x,t)|^dP(z,\tau)=||P||_{L^\infty(\Q_1)}|(x,t)|^d,$$
and we see that
$$||P||_{L^\infty(\Q_1)}|(x,t)|^d\leq C(\gamma+\lpl Q\rpo +\lpl u \rpo)+C''||P||_{L^\infty(\Q_1)}|(x,t)|^{d+\alpha}\,\,\,\mathrm{in}\,\,\,\Q_R.$$

Choosing $(x,t)$ small enough, this implies
$$||P||_{L^\infty(\Q_1)}\leq C(\gamma+\lpl Q\rpo +\lpl u \rpo).$$
Equivalently,
$$|P(x,t)|\leq C(\gamma+\lpl Q\rpo +\lpl u \rpo)|(x,t)|^d $$
establishing estimate \eqref{c1}.  Using the same argument on \eqref{psi}, we get \eqref{c2}, and \eqref{cu} follows from these two results.  The estimate \eqref{c3} follows from the interior estimates.

\end{proof}
\begin{thrm}\label{B}

For $\frac{Q}{2}+1<p<\infty$, let $u\in S^{2,1}_p(\Q_1)$ be a solution to $H_Au=f$ in $\Q_1$ satisfying hypothesis \eqref{hypot1} and \eqref{hypot2} and $f\in L^{p}(\Q_1)$.  Assume $d \geq 2$ and that $f$ and $u$ satisfy the following:

\begin{equation}\label{h1}
\limsup_{r\rightarrow 0} \frac{\lpl u\rpr}{r^{d+\frac{Q+2}{p}}} < \infty
\end{equation}and
\begin{equation}\label{h2}
\limsup_{r\rightarrow 0} \frac{\lpl f\rpr}{r^{d-2+\frac{Q+2}{p}}} < \infty.
\end{equation}
If for some $l\in \N$ and $\alpha \in(0,1)$, one has $f\in C_{p,d-2+l}^{\alpha}(0,0)$ and $a_{ij}\in C_{p,l}^{\alpha}(0,0)$,
then $u\in C_{\infty,d+l}^{\alpha}(0,0)$.  Moreover,
\begin{equation}
\sum_{|I|+2h=0}^{d+l}|X^{I}D_t^hu(0,0)|+[u]_{\infty,\alpha,d+l}\leq C\left(\lpl u\rpo    
      +\sum_{|J|+2m=d-2}^{d-2+l}|X^{J}D_t^mf(0,0)|+[f]_{p,\alpha,d-2+l}\right)
\end{equation}
where $C=C(G,p,d,l,\alpha,A)>0$.
\end{thrm}
\begin{proof}
The proof is given by induction on $l$.  If $l=0$, the result is given by Theorem \ref{big}.  To see this, first we need to show that we can apply the theorem by showing there exists a homogeneous polynomial $Q$ of degree $d-2$ such that $\lpl f-Q \rpr \leq \gamma r^{d-2+\al+\frac{Q+2}{p}}.$  The hypothesis on $u$ is immediate by taking $\beta=1$ in Theorem \ref{big}.

The assumption $f\in C_{p,d-2}^{\alpha}(0,0)$ means that there exists a polynomial of degree $d-2$, $Q,$ such that $$\lpl f-Q \rpr \leq [f]_{p,\al,d-2}\,\,r^{\al+d-2+\frac{Q+2}{p}},$$ and \eqref{h2} gives $\lpl f\rpr \leq Cr^{d-2+\frac{Q+2}{p}}.$  This information actually tells us that $Q$ is homogeneous of degree $d-2$. If $d=2$, there is nothing to show because $Q$ would be constant, so take $d>2$.  (Additionally, notice that the $d-2$ order Taylor expansion of $f$ satisfies the decay requirements.  We choose $Q$ to be the Taylor polynomial centered at the origin so that we can obtain the derivatives of $f$ at the origin. See Remark \ref{taylor}.)

\begin{eqnarray*}
\lpl Q\rpr &\leq& \lpl f-Q \rpr +\lpl f \rpr \\
&\leq&  C_1[f]_{p,\al,d-2}r^{d-2+\al+\frac{Q+2}{p}}+ C_2r^{d-2+\frac{Q+2}{p}} \\
&\leq& C[f]_{p,\al,d-2}r^{d-2+\frac{Q+2}{p}}.
\end{eqnarray*}
Using the familiar trick of writing $Q$ as a sum of homogeneous polynomials of degree $j$, gives
$$r^{j+\frac{Q+2}{p}}\lpl Q_j \rpo \leq C[f]_{p,\al,d-2}r^{d-2+\frac{Q+2}{p}}$$ for each $j=0,1,2,\ldots, d-2.$  Now in
$$\lpl Q_j \rpo \leq C[f]_{p,\al,d-2}r^{d-2-j},$$
the right hand side vanishes unless $j=d-2$ leaving $Q=Q_{d-2}$.  Applying Theorem \ref{big} with $\gamma=C[f]_{p,\al,d-2}$  gives the existence of $P_d,$ a homogeneous polynomial of degree $d$ such that $H_A(0)P_d=Q_{d-2}$ with the following properties:
\begin{equation}\label{yay1}
\lpl u\rpr \leq C\left(\lpl u\rpo +\lpl Q_{d-2}\rpo +[f]_{p,\al,d-2}\right)r^{d+\frac{Q+2}{p}},
\end{equation} 
\begin{equation}\label{yay2}
|P_d| \leq C\left(\lpl u\rpo +\lpl Q_{d-2}\rpo +[f]_{p,\al,d-2}\right)|(x,t)|^{d},
\end{equation} 
\begin{equation}\label{yay3}
|u-P_d| \leq C\left(\lpl u\rpo +\lpl Q_{d-2}\rpo +[f]_{p,\al,d-2}\right)|(x,t)|^{d+\al}, \text{ and}
\end{equation}
\begin{equation}\label{yay4}
\sum_{|I|=0}^{2}r^{|I|}\lpl X^I (u-P_d)\rpr \leq C\left(\lpl u\rpo +\lpl Q_{d-2}\rpo +[f]_{p,\al,d-2}\right)r^{d+\al+\frac{Q+2}{p}}.
\end{equation} 
Equation \eqref{yay3} gives $u\in C^\al_{\infty,d}(0,0)$ and $[u]_{\infty,\al,d}\leq C \left(\lpl u\rpo +\lpl Q_{d-2}\rpo +[f]_{p,\al,d-2}\right).$  Additionally, since we choose $Q_{d-2}$ to be the Taylor polynomial of $f$, $\lpl Q_{d-2} \rpo \leq C|X^{J}D_t^mf(0,0)|$ for $|J|+2m=d-2.$  
We have
\begin{equation}
\sum_{|I|+2m=1}^{2}|X^{I}D_t^mu(0,0)|+[u]_{\infty,\alpha,d}\leq C\left(\lpl u\rpo    
      +\sum_{|J|+2k=d-2}|X^{J}D_t^kf(0,0)|+[f]_{p,\alpha,d-2}\right).
\end{equation}
On the left hand side, we were able to add the derivatives of $u$ at the origin because the assumption on $u$ in fact implies that the the derivatives up to order $d$ are zero.

The case $l=1$ follows from Theorem \ref{big} in much the same way.  Begin by noticing that the assumptions on $f$ ensure the existence of homogeneous polynomials $Q_{d-2}$ and $Q_{d-1}$ with respective degrees $d-2$ and $d-1$ such that 
\begin{equation}
\lpl f-Q_{d-2}-Q_{d-1}\rpr \leq C[f]_{p,\alpha,d-1}(0,0)r^{d-1+\alpha+\frac{Q+2}{p}}.
\end{equation}

To see this, we notice from the $l=1$ assuptions, that there exists a polynomial $Q$ with degree at most $d-1$ such that 
\begin{equation}
\lpl f -Q\rpr \leq [f]_{p,\al,d-1}r^{d-1+\alpha+\frac{Q+2}{p}}
\end{equation}
From the hypothesis, $\lpl f\rpr \leq C r^{d-2+\frac{Q+2}{p}}.$

Write $Q$ as the sum of homogeneous polynomials $Q=\sum_{j=0}^{d-1}Q_j$, 
\begin{eqnarray*}
\lpl Q\rpr &\leq& \lpl f-Q \rpr +\lpl f \rpr \\
&\leq&  C_1[f]_{p,\al,d-1}r^{d-1+\al+\frac{Q+2}{p}}+ C_2r^{d-2+\frac{Q+2}{p}} \\
&\leq& C[f]_{p,\al,d-1}r^{d-2+\frac{Q+2}{p}}.
\end{eqnarray*}
Using the homogeneity and a change of variables,
$$r^{j+\frac{Q+2}{p}}\lpl Q_j \rpo \leq C[f]_{p,\al,d-1}r^{d-2+\frac{Q+2}{p}}$$ for each $j=0,1,2,\ldots, d-2.$  Now in the inequality
$$\lpl Q_j \rpo \leq C[f]_{p,\al,d-1}r^{d-2-j},$$
the right hand side vanishes unless $j=d-2$ or $j=d-1$ leaving $Q=Q_{d-1}+Q_{d-2}$.

Similarly, the hypothesis on $a_{ij}$ gives the existence of a homogeneous polynomial $a_{ij}^{(1)}$ of degree 1 such that
\begin{equation}
\lpl a_{ij}(x,t)-a_{ij}(0,0)-a_{ij}^{(1)}\rpr \leq C[a_{ij}]_{p,\alpha,1}(0,0)r^{1+\alpha+\frac{Q+2}{p}}\,\,\mathrm{for}\,\,i,j=1\ldots,m_1.
\end{equation}

From the $l=0$ case, we have the existence of a homogeneous polynomial $P_{d}$ of degree $d$ such that \eqref{c1} and \eqref{c2} hold and $H_A(0)P_d=Q_{d-2}$.  In particular, we get
\begin{equation}\label{c11}
|P_d(x,t)|\leq C\left(\left\|u\right\|_{L^{p}(\Q_{1})}+\left\|\mathrm{Q_{d-2}}\right\|_{L^{p}(\Q _1)}+[f]_{p,\alpha,d-2}(0,0)\right)|(x,t)|^{d}
\end{equation}
and
\begin{equation}\label{c22}
|u-P_d|(x,t)\leq C\left(\left\|u\right\|_{L^{p}(\Q_{1})}+\left\|\mathrm{Q_{d-2}}\right\|_{L^{p}(\Q _1)}+[f]_{p,\alpha,d-2}(0,0)\right)|(x,t)|^{d+\alpha}.
\end{equation}

Let $\psi=u-P_d$, and write $H_A\psi=H_Au-H_A(0)P_d+H_A(0)P_d-H_AP_d$.  Then $H_A\psi=f-Q_{d-2}+\sum_{i,j=1}^{m_1}(a_{ij}(x,t)-a_{ij}(0,0))X_iX_jP_d,$ or alternatively,
$$H_A\psi=f-Q_{d-2}-Q_{d-1}+Q_{d-1}+\sum_{i,j=1}^{m_1}(a_{ij}(x,t)-a_{ij}(0,0))X_iX_jP_d \equiv \tilde{f}.$$
Notice that $\tilde{Q}=Q_{d-1}+\sum_{i,j=1}^{m_1}a_{ij}^{(1)}X_iX_jP_d$ is homogeneous of degree $d-1$, and 
$\tilde{f}-\tilde{Q}=f-Q_{d-2}-Q_{d-1} +\sum_{i,j=1}^{m_1}(a_{ij}(x,t)-a_{ij}(0,0)-a_{ij}^{(1)})X_iX_jP_d$.

\begin{equation}
\lpl \tilde{f}-\tilde{Q} \rpr \leq ([f]_{p,\alpha,d-1}(0,0)+C\left\|P_d\right\|_{L^{\infty}(\Q_1)})r^{d-1+\alpha+\frac{Q+2}{p}}.
\end{equation}
This holds for all $r\leq 1$.  By \eqref{c3}, we have
$$\limsup_{r\rightarrow 0}\frac{\lpl \psi \rpr}{r^{d+\alpha+\frac{Q+2}{p}}}<\infty.$$

Now we can apply Theorem \ref{big} with $\psi$, $d+1$, $\tilde{f}$, and $\tilde{Q}$ replacing $u$, $d$, $f$, and $Q$ respectively to get a homogeneous polynomial $P_{d+1}$ of degree $d+1$ such that $H_A(0)P_{d+1}=\tilde{Q}$, and 
$$|P_{d+1}|(x,t)\leq C_{*}|(x,t)|^{d+1}\ \ for \ all \ (x,t)\in \Q_{1/2}$$
$$|\psi(x,t)-P_{d+1}(x,t)|\leq C_{*}|(x,t)|^{d+1+\alpha} \ \ for \ all \ (x,t)\in \Q_{1/2}$$
where $C_{*}\leq C([f]_{p,\alpha,d-1}+\left\| \psi \right\|_{L^p(\Q_{1/2})}+\left\|P_d\right\|_{L^{\infty}(\Q_1)}+\lpl \tilde{Q}\rpo)$.
Using the expressions for $\psi$ and $\tilde{Q}$, we have
$$C_{*}\leq C([f]_{p,\alpha,d-1}+\left\| u-P_d \right\|_{L^p(\Q_{1/2})}+\left\|P_d\right\|_{L^{\infty}(\Q_1)}+\lpl Q_{d-1}\rpo).$$
We can eliminate $P_d$ by using 
\eqref{c11} and \eqref{c22} to notice that on $\Q_1$
$$\left\|P_d\right\|_{L^{\infty}(\Q_1)}\leq C(\lpl u\rpo +\lpl {Q_{d-2}} \rpo + [f]_{p,\alpha,d-2}).$$
Then
$$C_{*}\leq C([f]_{p,\alpha,d-1}+[f]_{p,\alpha,d-2}(0,0)+\left\| u \right\|_{L^p(\Q_{1})}+\lpl Q_{d-2}\rpo+\lpl Q_{d-1}\rpo).$$  Finally,
$$C_{*}\leq C([f]_{p,\alpha,d-1}+\left\| u \right\|_{L^p(\Q_{1})}+\sum_{|J|+2k=d-2}^{d-1}|X^{J}D_t^kf(0,0)|).$$
This completes the proof for $l=1$.
\end{proof}

\begin{thrm}
For $Q+2<p<\infty,$ let $u\in S_{p}^{2,1}(\Q_1)$ be a solution of $H_Au=f$ in $\Q_1\subset \mathbb{G}\times \R$ with $f\in L^{p}(\Q_1).$  Assume, for some 
$\alpha\in(0,1)$ and integer $\tilde{d}\geq 2$, $f\in C_{p,\tilde{d}-2}^{\alpha}(0,0)$ and $a_{ij}\in C_{p,\tilde{d}-2}^{\alpha}(0,0)$.  
Then $u\in C_{\infty,\tilde{d}}^{\alpha}(0,0)$ and 
\begin{equation}
\left\|u\right\|_{\infty,\alpha,\tilde{d}}\leq C(\lpl u\rpo + \left\|f\right\|_{p,\alpha,\tilde{d}-2})
\end{equation}
where $C=C(G,p,\tilde{d},l,\alpha,A)>0.$
\end{thrm}

\begin{proof}
Let $P_*$ be the $1$st order Taylor polynomial expansion of $u$ at the origin.  We can apply Theorem \ref{PSTI} and get 
\begin{eqnarray*}
\frac{|u(x,t)-P_*(x,t)|}{|(x,t)|^{1+\alpha}}&\leq& C |(x,t)|^{-\alpha} \sup\limits_{\stackrel{|(z,\tau)|\leq b^k|(x,t)|}{i=1,\ldots,m_1}} |X_iu(z,\tau)-X_iu(0,0)| \\
&=&C\left\|X_iu\right\|_{\Gamma^{\alpha}(\Q_{1/2})}.
\end{eqnarray*}
By Theorem \ref{mce2}, $S_p^{2,1}(\Q)\subset\Gamma_{loc}^{\alpha '}(\Q)$ for $\alpha '=2-\frac{Q+2}{p}$.  Here we must account for the derivative and get $\alpha=1-\frac{Q+2}{p}$, hence
\begin{eqnarray*}
\left\|X_iu\right\|_{\Gamma^{\alpha}(\Q_{1/2})}&\leq& C\left\|u\right\|_{S_p^{2,1}(\Q_{1/2})}\\
&\leq&C(\lpl u\rpo +\lpl f\rpo )
\end{eqnarray*}
for all $|(x,t)|<1/2.$ The previous line is obtained by using the estimates of Lemma \ref{Lp2}.  
Observe by the definition of Taylor polynomial and Lemma \ref{Lp2}, we also have 
$$|P_{*}(0,0)|+\sum_{i=1}^{m_1}|X_iP_{*}(0,0)|\leq C(\lpl u\rpo +\lpl f\rpo ).$$
The goal now is to successively apply Theorems \ref{big} and \ref{B} with $d=2$ and $u=u-P_*$.  To see that the hypotheses of Theorem \ref{big} are satisfied, first notice that $H_A(u-P_*)=f$ because $P_*$ is of order $1$ and is annihilated by $H_A$.  By the above arguments, we know that $|u-P_*|\leq \tilde{C}r^{1+\alpha}=\tilde{C}r^{2-\frac{Q+2}{p}}$.  Using this and a dyadic decomposition, we can get that 
$\lpl u-P_* \rpr <Cr^{2}$ for all $r\in (0,1).$
\begin{eqnarray*}
\lpl u-P_* \rpr&=&\left(\int_{\Q_r} |u-P_*|^{p}dxdt \right)^{1/p} \\
&\leq &\left(\int_{\Q_r}(\tilde{C}|(x,t)|^{2-\frac{Q+2}{p}}) ^{p}dxdt\right)^{1/p} \\
&\leq& \tilde{C}\sum_{j=0}^{\infty}\left(\int_{\Q_{\frac{r}{2^{j}}}/\Q_{\frac{r}{2^{j+1}}}}|(x,t)|^{2p-Q-2}dxdt\right)^{1/p}\\
&\leq& \tilde{C}\sum_{j=0}^{\infty}\left(\frac{r}{2^{j+1}}\right)^{2-\frac{Q+2}{p}}\left(\frac{r}{2^{j}}\right)^{\frac{Q+2}{p}}\\
&\leq& \tilde{C}r^{2}\sum_{j=0}^{\infty}\left(\frac{1}{2^j}\right)^{2}\\
&\leq& Cr^{2}.
\end{eqnarray*}

Since $\frac{Q+2}{p}<1$,  
$$\frac{\lpl u-P_* \rpr}{r^{1+\beta+\frac{Q+2}{p}}}\leq \frac{Cr^2}{r^{1+\beta+\frac{Q+2}{p}}}\leq Cr^{1-\beta-\frac{Q+2}{p}}\rightarrow 0\ \ \mathrm{as}\ \ r\rightarrow 0$$ for some $\beta$ chosen small enough.
Since $f\in C_{p,0}^{\alpha}(0,0)$, $f\in \Gamma^{\alpha}(0,0)$ by Proposition \ref{mce1}, and so $|f|\leq Cr^{\alpha}$ on $\Q_r$. Then 
$$\lpl f \rpr \leq Cr^{\alpha+\frac{Q+2}{p}}.$$  
This is exactly what is needed to satisfy hypotheses in Theorem \ref{big}. 
The conclusion 
\begin{equation}
\lpl u-P_* \rpr \leq Cr^{2+\frac{Q+2}{p}}
\end{equation}
gives enough of a decay gain to satisfy hypothesis on $u-P_*$ in Theorem \ref{B} where $d=2$.  We also need to satisfy the hypothesis on $f$.  To do this, we once again use $$\lpl f \rpr \leq Cr^{\alpha+\frac{Q+2}{p}}\leq Cr^{\frac{Q+2}{p}}.$$ 
Now Theorem \ref{B} reads that if $f\in C_{p,l}^{\alpha}(0,0)$ and $a_{ij}\in C_{p,l}^{\alpha}(0,0)$, then $u-P_*\in C_{\infty,2+l}^{\alpha}(0,0)$.  Moreover,
\begin{eqnarray}
\sum_{|I|+2h=1}^{2+l}|X^{I}D_t^h(u-P_*)(0,0)|+[u-P_*]_{\infty,\alpha,2+l}&\leq& C(\lpl u-P_*\rpo+ \nonumber \\
&&\sum_{|J|+2m=0}^{l}|X^{J}D_t^mf(0,0)|+[f]_{p,\alpha,l}(0,0))
\end{eqnarray}
 Let $l=\tilde{d}-2$ to get the final result. 
\end{proof}

\end{document}